\documentclass[lettersize,journal]{IEEEtran}
 
\usepackage{enumerate}
\usepackage{comment}
\usepackage{graphicx}
\usepackage{dsfont}

\usepackage{amsthm}

\usepackage{amsmath,amsfonts,amssymb}
\newtheorem{definition}{Definition}
\newtheorem{remark}{Remark}

\newtheorem{lemma}{Lemma}
\usepackage{mathtools}
\newtheorem{problem}{Problem}
\newtheorem{assumption}{Assumption}

\newtheorem{proposition}{Proposition}
\newtheorem{theorem}{Theorem}
\newtheorem{corollary}{Corollary}

 \usepackage{soul}

\usepackage{tikz}
\usetikzlibrary{arrows,chains,matrix,positioning,scopes}
\makeatletter
\tikzset{join/.code=\tikzset{after node path={%
\ifx\tikzchainprevious\pgfutil@empty\else(\tikzchainprevious)%
edge[every join]#1(\tikzchaincurrent)\fi}}}
\makeatother
\tikzset{>=stealth',every on chain/.append style={join},
         every join/.style={->}}
\tikzstyle{labeled}=[execute at begin node=$\scriptstyle,
   execute at end node=$]

\newcommand{\myunit}{0.6 cm}
\tikzset{
    node style sp/.style={draw,circle,minimum size=\myunit},
    node style ge/.style={circle,minimum size=\myunit},
    arrow style mul/.style={draw,sloped,midway,fill=white},
    arrow style plus/.style={midway,sloped,fill=white},
}

\definecolor{matlabBlue}{rgb}{0,0.4470, 0.7410}

\definecolor{matlabBlue2}{rgb}{0.7,0.2, 0}

\definecolor{matlab1}{rgb}{0,0.4470, 0.7410}
\definecolor{matlab2}{rgb}{0.8500, 0.3250, 0.0980}
\definecolor{matlab3}{rgb}{0.9290, 0.6940, 0.1250}
\definecolor{matlab4}{rgb}{0.4940, 0.1840, 0.5560}
\definecolor{matlab5}{rgb}{0.4660, 0.6740, 0.1880}
\definecolor{matlab6}{rgb}{0.3010, 0.7450, 0.9330}

\usepackage{algpseudocode}
\usepackage{algorithm}
\newcommand\NoDo{\renewcommand\algorithmicdo{}}

\usepackage{subcaption}
\usepackage{cite}

\title{ 
Learning in Memristive Neural Networks
}

\author{H. M. Heidema, H. J. van Waarde, B. Besselink% <-this % stops a space
\thanks{The authors are with the Jan C. Willems Center for Systems
and Control, and the Bernoulli Institute for Mathematics, Computer
Science, and Artificial Intelligence, University of Groningen, The Netherlands. 
Marieke Heidema and Bart Besselink are also with CogniGron (Groningen Cognitive Systems and Materials Center), University of Groningen, The Netherlands. Marieke Heidema and Bart Besselink acknowledge the financial support of the CogniGron research center and the Ubbo Emmius Funds. 
Henk van Waarde acknowledges financial support by the Dutch Research Council under the NWO Talent Programme Veni Agreement (VI.Veni.222.335). Email: {\tt\small h.m.heidema@rug.nl; 
h.j.van.waarde@rug.nl; b.besselink@rug.nl.}}
}

\begin{document}
\mathtoolsset{showonlyrefs}

\maketitle
\thispagestyle{empty}

\begin{abstract}
Memristors are nonlinear two-terminal circuit elements whose resistance at a given time depends on past electrical stimuli. 
Recently, networks of memristors have received attention in neuromorphic computing since they can be used to implement Artificial Neural Networks (ANNs) in hardware.
For this, one can use a class of memristive circuits called crossbar arrays.
In this paper, we describe a circuit implementation of an ANN and resolve three questions concerning such an implementation.  
In particular, we show (1) how to evaluate the implementation at an input, 
(2) how the resistance values of the memristors at a given time can be determined from external (current) measurements, and (3) how the resistances can be steered to desired values by applying suitable external voltages to the network. 
The results will be applied to two examples: an academic example to show proof of concept and an ANN that was trained using the MNIST dataset. 
 
\end{abstract}

\begin{IEEEkeywords}
Artificial neural networks, memristors, crossbar arrays, memristive electrical circuits, inference, writing.
\end{IEEEkeywords}

\section{INTRODUCTION}
Over the past decades, our society and technology has become increasingly reliant on computing technology. However, current digital computers use a lot of energy, mainly because they require the constant transfer of data between memory and processing units.
To reduce the energy consumption of data transmissions, we require a different type of computer with the ability to perform in-memory computing, i.e., perform certain computations directly in the memory where the data is stored. 
To this end, the potential of neuromorphic computing has been studied \cite{Ribar2021, Indiveri2015}. 
This is a method of designing \textit{analog} information processing systems by drawing inspiration from the nervous system in the brain. 
Such computing technology could lead to a significant reduction in energy consumption,  compared to existing digital computers, see \cite{Mead1990}.

In particular, there has been an increasing interest in efforts to implement Artificial Neural Networks (ANNs) directly in analog hardware, e.g., \cite{Dongare2012, Rezvani2012, Hopfield1988}.
These networks are inspired by biological neural networks and consist of layers of neurons between which information travels via synapses.
They form the foundation for modern AI applications and have been used successfully in tasks ranging from pattern recognition, image processing, to speech recognition and language processing  \cite{Conti1994, Hassoun1995, Deisenroth2020}.
Implementing them directly in analog hardware, as opposed to their software implementation on digital computers, could lead to more energy-efficient computing technology which is able to perform in-memory computing.

To this end, devices that behave like synapses are needed \cite{Thomas2013, Snider2011, Khalid2019, Sah2014}. 
Memristors, originally introduced by Chua \cite{Chua1971} as the `fourth electrical circuit element', are suitable candidates for such synapses.
They can be regarded as tunable resistors for which the resistance value depends on past external stimuli and that, in the absence of external stimuli, retain their resistance value.  When viewing this resistance value as the memory of a memristor, a change in resistance can be seen as learning. This mimics the role of synapses in ANNs, whose memory is their weight which gets updated during training.

In particular, memristors can naturally be used as synapses in hardware implementations of ANNs when they are organized in so-called crossbar arrays \cite{Bayat2018}. Crossbar arrays are electrical circuits consisting of row and column bars, with a memristor on every cross-point. 
Crossbar arrays have been studied for resistors \cite{Sun2019} and memristors \cite{Sebastian2020,Hong2023,Mannocci2021}. Here, it has been found that this network structure can be used to compute matrix-vector products, where the (instantaneous) resistance values of the elements correspond to the entries of the matrix.
In the case of a resistive crossbar array, this means that the matrix with which we can perform matrix-vector products is fixed. However, in the case of memristive crossbar arrays, the resistance value of the memristors is tunable, leading to a set of matrices with which matrix-vector products can be computed. 
ANNs can be implemented in hardware by interconnecting multiple memristive crossbar arrays. Here, the crossbar arrays correspond to synaptic interconnections of the ANN and the interconnection between arrays corresponds to the neurons in the ANN.
For such a hardware implementation of an ANN, the tunability property of the memristors corresponds to the tunability of the strength of the synapses.
This effectively leads to a set of ANNs that can be implemented in the same electrical circuit.

Memristors and memristive networks have been studied previously, \cite{Corinto2016}, \cite{Corinto2021}, where the notions of charge and flux-controlled memristors were introduced, and their passivity properties were established. Furthermore, past research on memristive networks includes investigations on their monotonicity \cite{Chaffey2024} and how this relates to their passivity \cite{Corinto2015}, and found that memristive networks can equivalently be modeled as a single memristor \cite{Huijzer2025}.  
Previous research on memristive crossbar arrays presented a way to retrieve the resistance values of memristors in an array and to steer them so that they coincide with the entries of a given matrix \cite{Heidema2024}. 
These studies present a way of modeling electrical circuits with memristors, using graph theory and Kirchhoff's laws. 
In this work, we will use this to analyze the interconnection of multiple memristive circuits. 

Note that an ANN essentially represents a function, which one could evaluate at a given input. 
Now, what is missing from the previous research is
a way to translate inputs to the function to (voltage) input signals to the hardware implementation and to translate the corresponding output (current) signals to the function outputs. 
In particular, the current literature lacks generality as they focus on particular types of physical memristors \cite{Adhikari2015,Zhang2018,Jiang2025}. 
Methods are needed for general classes of memristors to evaluate the hardware implementation for a given input in a \textit{non-invasive} way: in such a way that the resistance values, which change when external input is applied due to the nature of a memristor, are the same before and after the evaluation.
Furthermore, it needs to be investigated how to steer the resistance values of general types of memristors in the hardware implementation of the ANN so that they coincide with some desired values for the strength of the synapses. 
To this end, we define ANNs, introduce memristors and memristive crossbar arrays, and use this to describe a circuit implementation of an ANN based on memristive crossbar arrays. 
The contributions of this paper are then as follows: (1) we show how one can evaluate (\textit{infer}) the implementation at an input; (2) we show how one can determine (\textit{read}) the resistance values of the memristors at a given time in such a circuit, based on current measurements; (3) we develop and prove convergence of an algorithm that designs input signals to the circuit to steer (\textit{write}) the resistance values of the memristors in such a circuit to desired values. 
Together, these results allow us to implement ANNs in hardware and to compute with them. 

 The remainder of this paper is organized as follows. In Section~\ref{sec: MNNs}, we introduce ANNs, memristors, and memristive crossbar arrays. Furthermore, we discuss how to implement an ANN in hardware using memristive crossbar arrays.
 Section~\ref{sec: problem statement} then formalizes the research problems. 
 Section~\ref{sec:inference} will then concern evaluating the circuit implementation at an input.
 Section~\ref{sec:reading} deals with reading the instantaneous resistance (conductance) values of the memristors. 
 Section~\ref{sec:writing} discusses writing the resistance (conductance) values of the memristors. 
 Section~\ref{sec: applications} will be concerned with two applications of our theory. 
 What follows is a conclusion and discussion of the results in Section~\ref{sec: conclusion}.

\textit{Notation}:
The column vector of size $m$ with all ones is denoted by $\mathds{1}_m$. 
The set of integers from 1 to $n$ is denoted by $[n]$, i.e. $[n]=\{1,\ldots,n\}$. 
The Kronecker product of two matrices $A\in\mathds{R}^{m\times n}$ and $B\in\mathds{R}^{p\times q}$ is denoted by $A\otimes B \in\mathds{R}^{mp\times nq}$. 
The Hadamard (element-wise) product of two matrices $A,B\in\mathds{R}^{m\times n}$ is denoted by $A\odot B \in\mathds{R}^{m\times n}$.
A function $f:\mathds{R}\rightarrow\mathds{R}^{m\times n}$ is called even (odd) if $f(x) = (-)f(-x)$ for all $x\in\mathds{R}$. A function $f:\mathds{R}\rightarrow\mathds{R}^{m\times n}$ is called even (odd) on the interval $[a,b]$ if $f(x) = (-)f(a+b-x)$ for all $x\in[a,\tfrac{a+b}{2}]$.  
A function $g:[a,b]\subseteq\mathds{R}\rightarrow\mathds{R}$ is called strictly monotone if
    \begin{equation}
        (g(x)-g(y)) (x-y)>0
    \end{equation}
for all distinct $x,y\in[a,b]$. A function $g:\mathds{R}\rightarrow\mathds{R}$ is called $\beta-$Lipschitz continuous if there exists $\beta>0$ such that
    \begin{equation}
        |g(x)-g(y)| \leq \beta |x-y|
    \end{equation}
for all $x,y\in\mathds{R}$.

\section{MEMRISTIVE NEURAL NETWORKS} \label{sec: MNNs}
\subsection{Artificial neural networks}
In this paper, we consider feedforward artificial neural networks. 
Specifically, we consider a network of $L+1$ layers with $n_\ell\in\mathds{N}$ neurons in layer $\ell\in [L]$. 
The first layer represents external inputs $\hat{u}\in\mathds{R}^{n_0}$ given to the network and the last layer represents the network's outputs $\hat{y}\in\mathds{R}^{n_L}$. The layers in between, so-called \textit{hidden layers}, process information layer by layer via the synapses connecting the layers. In particular, the activity of the neurons in layer $\ell$ is determined by activities of the neurons in layer $\ell-1$, specifically,
\begin{align} \label{eqn: neuron eqns}
    a^0 &= \hat{u}, \quad a^\ell = \sigma^\ell(\hat{M}^\ell a^{\ell-1}), \quad \hat{y} = a^L,
\end{align}
where $\hat{M}^\ell\in\mathds{R}^{n_\ell\times n_{\ell-1}}$ is a matrix of synaptic \textit{weights}, which represent the strength of the synapses connecting neurons in layer $\ell-1$ with neurons in layer $\ell$.
Moreover, $\sigma^\ell:\mathds{R}^{n_\ell}\rightarrow\mathds{R}^{n_\ell}$ denotes the element-wise application of the \textit{activation function} $\sigma:\mathds{R}\rightarrow \mathds{R}$, which is assumed to be the same for each neuron in the network. Note that the activity $a^\ell$ is the output of layer $\ell$ and that the output of the network is given by the output of the last layer of neurons, i.e., $\hat{y} = a^L$. 
Now, for given synaptic weight matrices $\hat{M}^1,\ldots,\hat{M}^L$, this describes a function from $\hat{u}$ to $\hat{y}$, denoted as 
\begin{equation} \label{eqn: ANN function}
    \hat{y} = h(\hat{u}).
\end{equation} 

\begin{assumption}\label{assumption 1}
    The activation function $\sigma(\cdot)$ is continuously differentiable, odd, strictly monotone, and $\eta$-Lipschitz continuous.  
\end{assumption}

\begin{remark} \label{remark 1}
    Assumption~\ref{assumption 1} includes frequently used activation functions such as the hyperbolic tangent, 
    see \cite{Prezioso2015} and \cite{Zamanidoost2015}, and other scaled versions of the sigmoid function. Note that feedforward neural networks, in general, also contain so-called bias terms \cite{Arnold2019}.
    Here, for simplicity, the bias of each neuron is taken to be zero.
\end{remark}

\subsection{Memristors}
Artificial neural networks can be implemented in hardware using memristors and memristive crossbar arrays \cite{Xia2019}.
In this paper, we consider  memristors as originally postulated by Chua \cite{Chua1971}, see Figure~\ref{fig: single memristor}. They pose a relation between the (magnetic) flux $\varphi$ and (electric) charge $q$, satisfying
\begin{equation}
    \tfrac{d}{dt} \varphi(t) = V(t), \quad \tfrac{d}{dt} q(t) = I(t),
\end{equation}
with $V$ being the voltage across and $I$ the current through the memristor.
In particular, we consider flux-controlled memristors of the form
\begin{align}
     q(t) &= g(\varphi(t)), \label{eqn:flux controlled}
\end{align} 
for some function $g:\mathds{R}\rightarrow\mathds{R}$.
For ease of notation, we will omit the argument $t$ in what follows.

\begin{figure}[t]
    \centering
     \includegraphics[width=0.25\textwidth]{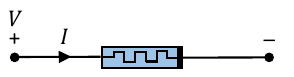}
    \caption{Memristor with current $I$ through it and voltage $V$ over it.}
    \label{fig: single memristor}
\end{figure}

\begin{assumption}\label{assumption 2}
    The function $g(\cdot)$ is continuously differentiable and strictly increasing, i.e.,
    \begin{equation}
        W(\varphi):= \tfrac{d g(\varphi)}{d\varphi}>0
    \end{equation}
for all $\varphi\in\mathds{R}$. 
\end{assumption}

Note that Assumption~\ref{assumption 2} implies that the memristor in \eqref{eqn:flux controlled v2} is passive, see Theorem 6 in \cite{Chua1971}. Furthermore, we call $W(\varphi)$ the \textit{memductance} of the memristor
and note that, by Assumption~\ref{assumption 2}, the memductance is always positive.  
Then, as follows from time-differentiation of \eqref{eqn:flux controlled}, the following dynamical system describes the memristor:
\begin{equation}
\begin{aligned}
    \tfrac{d}{dt} \varphi &= V,  \quad
   I = W(\varphi) V. \label{eqn:flux controlled v2}
\end{aligned}
\end{equation} 
From the dynamics~\eqref{eqn:flux controlled v2}, it is clear that the memductance value of the memristor changes when external (voltage) stimuli are applied.
In particular, we say that the memristor is \emph{tunable} as its memductance value can be steered by choosing an appropriate input voltage.  Moreover, the memristor has \textit{memory} in the sense that its memductance value can be retained by supplying no input voltage thereafter. 

\begin{assumption}\label{assumption W strict monotone}
    The function $W(\cdot)$ is strictly monotone and $\beta-$Lipschitz continuous. Furthermore, there exist $W_\mathrm{max}\geq W_\mathrm{min}>0$ such that
    $$W_\mathrm{max} \geq W(x) \geq W_\mathrm{min}, \forall x\in\mathds{R}. $$
\end{assumption}
\vspace{0.2cm}

\begin{remark} 
    Assumptions~\ref{assumption 2} and \ref{assumption W strict monotone} are both satisfied by the HP memristor \cite{Strukov2008}, which is a frequently used example of a physical memristor \cite{Buscarino2012}, \cite{Olumodeji2017}.  
\end{remark}
\subsection{Memristive crossbar arrays}  
Having defined memristors, let us now consider a network of interconnected memristors. In particular, we consider a crossbar array with $n$ rows and $m$ columns, as seen in Figure~\ref{fig:memristive crossbar array}. 
At the point where row $k$ and column $j$ cross, there is a memristor  
satisfying Assumptions~\ref{assumption 2} and \ref{assumption W strict monotone}, and an associated switch. 
In the remainder of this paper we will, for simplicity,  assume all memristors to be the same, i.e. they are described by the same function $g(\cdot)$ as in \eqref{eqn:flux controlled} and hence have the same dynamics~\eqref{eqn:flux controlled v2}. 
    Note that the results presented in this paper can readily be extended to the case in which all memristors are different.

\begin{figure}[ht]
    \centering
    \begin{subfigure}[t]{0.5\textwidth}
        \centering
        \includegraphics[width=0.95\textwidth]{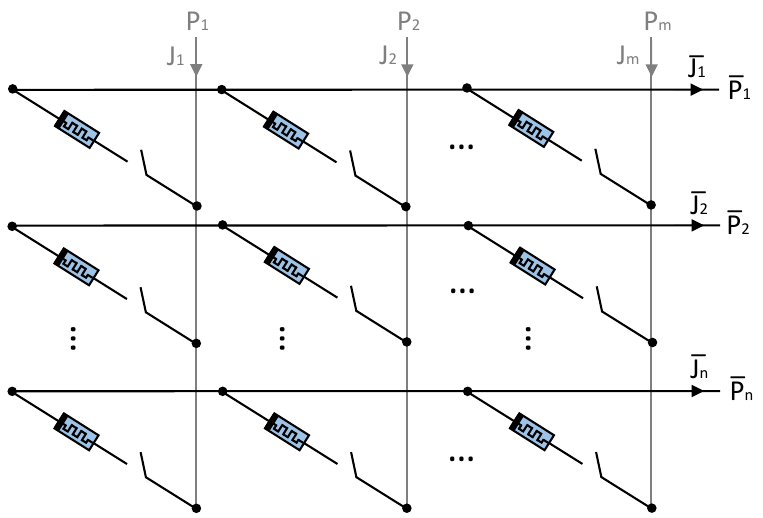}
        \caption{A three-dimensional representation}
        \label{fig:}
    \end{subfigure}
    \begin{subfigure}[t]{0.5\textwidth}
        \centering
        \includegraphics[width=0.4\textwidth]{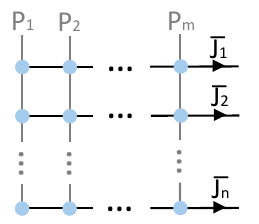}
        \caption{An abstract two-dimensional representation}
    \end{subfigure}
    \caption{A schematic depiction of a memristive crossbar array with switches.}
    \label{fig:memristive crossbar array}
\end{figure}

To model a memristive crossbar array, we will denote the voltage across, the current through, and the flux of the $(k,j)$-th memristor by  $V_{kj}$,  $I_{kj}$, and $\varphi_{kj}$, respectively.
The voltages associated with the memristors in the array can then be collected in a matrix as 
\begin{align}
   V &= \begin{bmatrix} V_{11} &  \hdots & V_{1m} \\
   \vdots & \ddots & \vdots \\
   V_{n1}   &\cdots & V_{nm} \end{bmatrix}  
\end{align}
and, similarly,  the currents and fluxes are collected in the matrices $I$ and $\varphi$. In addition, with some abuse of notation, we denote the matrix of memductance values by
\begin{align}
    {W}(\varphi) 
    &=\begin{bmatrix} W(\varphi_{11}) &  \hdots & W(\varphi_{1m}) \\
   \vdots & \ddots & \vdots \\
   W(\varphi_{n1})   &\cdots & W(\varphi_{nm}) \end{bmatrix}.
\end{align}
Using \eqref{eqn:flux controlled v2}, we collect and write the dynamics of the memristors in the array as 
\begin{equation}
\begin{aligned}
    \tfrac{d}{dt} \varphi &= V, \quad    I=W(\varphi) \odot V, \label{eqn: crossbar array}
\end{aligned}
\end{equation}
with $\odot$ denoting the Hadamard product.

Furthermore, we consider voltage potentials and currents at the terminals of our crossbar array, which we respectively denote by $P_j$ and $J_j$ at the beginning of the column bars, and $\Bar{P}_k$ and $\Bar{J}_k$ at the end of the row bars. Now,  we can collect the voltage potentials and currents at the column and row terminals into vectors $P,J\in\mathds{R}^m$ and $\Bar{P}, \Bar{J}\in\mathds{R}^n$, and 
let $\Tilde{P}\in\mathds{R}^{m+n}$ and $\Tilde{J}\in\mathds{R}^{m+n}$ respectively denote the combined vector of voltage potentials and currents at the terminals, i.e., 
\begin{align} 
    \Tilde{P} &= \begin{bmatrix}
        P_1 \ \cdots \ P_{m} \ \lvert \  \Bar{P}_1 \ \cdots  \ \Bar{P}_{n}
    \end{bmatrix}^\top = \begin{bmatrix}
        P^\top \ \lvert \ \Bar{P}^\top
    \end{bmatrix}^\top, \label{eqn: P^l} \\
     \Tilde{J} &= \begin{bmatrix}
        J_1 \ \cdots \ J_{m} \ \lvert \  -\Bar{J}_1 \ \cdots  \ -\Bar{J}_{n}
    \end{bmatrix}^\top = \begin{bmatrix}
        J^\top \ \lvert \ -\Bar{J}^\top
    \end{bmatrix}^\top. \label{eqn: J^l}
\end{align}
In addition, let $S$  be the switch matrix associated with the array, given by
\begin{align}
    S &= \begin{bmatrix} S_{11} &  \hdots & S_{1m} \\
   \vdots & \ddots & \vdots \\
   S_{n1}   &\cdots & S_{nm} \end{bmatrix}  
\end{align} 
 where $S_{kj}\in\{0,1\}$ is defined as 
\begin{align} \label{eqn: switch variables}
    S_{kj} &= \left\{
	\begin{array}{ll}
		1  & \mbox{if the $(k,j)$-th switch is closed}, \\
		0 & \mbox{if the $(k,j)$-th switch is open},
	\end{array}
\right. 
\end{align} 
for all $k\in[n]$ and $j\in[m]$.

Kirchhoff's current and voltage law tell us that 
\begin{equation}
\begin{aligned}
    \Tilde{J} &=  \begin{bmatrix}
        J \\ -\Bar{J}
    \end{bmatrix}= \begin{bmatrix}
        (S\odot I)^\top \mathds{1}_n \\ -(S\odot I)\mathds{1}_m
    \end{bmatrix}, \\
    V&=S \odot (\mathds{1}_n P^\top - \Bar{P}\mathds{1}_m^\top), \label{eqn: Kirchhoff's laws}
\end{aligned}
\end{equation} 
respectively. 
Combining the dynamics~\eqref{eqn: crossbar array} with Kirchhoff's laws \eqref{eqn: Kirchhoff's laws}, the dynamics of the array can be written as 
\begin{equation}\label{eqn:general crossbar array}
\begin{array}{ll}
    \tfrac{d}{dt} \varphi &= S \odot (\mathds{1}_n P^\top - \Bar{P} \mathds{1}_m^\top), \\
    \Tilde{J}&= \begin{bmatrix}
        (W(\varphi)\odot S\odot  (\mathds{1}_n P^\top-\Bar{P}\mathds{1}_m^\top))^\top \mathds{1}_n \\ -(W(\varphi)\odot S\odot (\mathds{1}_n P^\top-\Bar{P}\mathds{1}_m^\top)) \mathds{1}_m
    \end{bmatrix}. 
\end{array}
\end{equation}
Note that, to obtain the second equation, we used that $S\odot S = S$, by definition of $S$.

Now, note that a crossbar array implements matrix-vector multiplication \cite{Heidema2024}. In particular, assuming that $\Bar{P}\!=\! 0$ it follows from \eqref{eqn:general crossbar array} that, for input voltage $P$ to the column bars, the measured output current $\Bar{J}$ at the row bars is given by
\begin{align}
    \Bar{J} &=  ( W(\varphi)\odot S\odot(\mathds{1}_{n} P^\top) \mathds{1}_n = ( W(\varphi)\odot S)P. \label{eqn: Jbar = MP}
\end{align}
In the remainder of this paper, we will consider the row bars of the crossbar array to be grounded, meaning that 
$\Bar{P} = 0$. 
The dynamics~\eqref{eqn:general crossbar array} then becomes
\begin{align} \label{eqn:general crossbar array v2}
    \tfrac{d}{dt} \varphi &= S \odot (\mathds{1}_n P^\top), \quad \Bar{J} = ( W(\varphi)\odot S)P.
\end{align} 
Here, we can view \eqref{eqn:general crossbar array v2} as a dynamical system with input $P$, state $\varphi$, and output $\Bar{J}$.
For given switch settings, we denote the state trajectory of \eqref{eqn:general crossbar array v2}, at time $t$, for initial condition $\varphi(0)=\varphi_0$ and input voltages $P:[0,\infty)\rightarrow \mathds{R}^{m}$ as  $\varphi(t; \varphi_0,P)$. The corresponding (output) currents are denoted by $\Bar{J}(t; \varphi_0,P)$.

\subsection{Circuit implementation of ANN} \label{sec:hardware implementation}
\begin{figure*}
    \centering
    \includegraphics[width=0.95\textwidth]{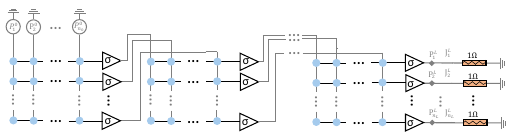}
    \caption{Neural network implementation.}
    \label{fig:ANN}
\end{figure*}
The structure of a feed-forward artificial neural network can naturally be captured by the interconnection of multiple crossbar array circuits \cite{Bayat2018}, as depicted in Figure~\ref{fig:ANN}. Here, we used the schematic representation of memristive crossbar arrays as in Figure~\ref{fig:memristive crossbar array}. In particular, memristive crossbar arrays can be used to implement synapses and Current-Controlled Voltage Sources (CCVSs) can be used to implement neurons. 
Hereto, consider the interconnection of $L$ memristive crossbar arrays as seen in Figure~\ref{fig:ANN}, where the $\ell$-th array has $n_{\ell}$ rows and $n_{\ell-1}$ columns for $\ell\in[L]$. Here, each synaptic weight $\hat{M}_{kj}^\ell$ is represented by the instantaneous memductance value of a memristor as 
\begin{equation} \label{eqn:wij^{l}}     
\hat{M}_{kj}^\ell  = W(\varphi_{kj}^{\ell}), \end{equation} for all $k\in [n_\ell]$ and $j\in [n_{\ell-1}]$. 
Furthermore, the activation function is implemented by a current-controlled voltage source that, based on the current $\bar{J}^\ell_{k}$ at the end of the $k$-th row bar of the $\ell$-th crossbar array, gives voltage $P_k^\ell=\sigma(\bar{J}^\ell_{k})$ to the $k$-th column of the $(\ell+1)$-th crossbar array, for all $\ell\in[L]$ and $k\in [n_\ell]$. 
In addition, the current-controlled voltage source ensures that $\Bar{P}^\ell=0$ for all $\ell\in[L-1]$. 
From \eqref{eqn:general crossbar array v2}, the dynamics of the interconnection of the $L$ crossbar arrays can be derived as
\begin{equation}
\begin{aligned} \label{eqn: hardware ANN}
    P^0 &= u, \\
    \tfrac{d}{dt} \varphi^\ell &= S^\ell \odot (\mathds{1}_{n_\ell} (P^{\ell-1})^\top), \\
    \bar{J}^\ell &= (W(\varphi^\ell)\odot S^\ell) P^{\ell-1}, \quad \ell\in\{1,\ldots,L\},
    \\
    P^\ell &= \sigma^\ell(\bar{J}^\ell), \\
    y &= P^L.
\end{aligned}
\end{equation} 
Here, $\varphi^\ell$ and $S^\ell$ denote the flux and switch matrices associated with the $\ell$-th crossbar array, respectively. 
The hardware implementation of the network hence describes a function from $u$ to $y$. 
 Here, note that the inputs to the network are represented by voltage sources $P_1^0,\ldots,P_{n_0}^0$ and the corresponding outputs are the voltage potentials $P_1^L,\ldots,P_{n_L}^L$. 
 The resistors with a resistance of $1\Omega$, see Figure~\ref{fig:ANN}, are such that $J_k^L$ is equal in value to $P_k^L$ for all $k\in[n_L]$. In this paper, we assume that the currents $J_1^L, \ldots, J_{n_L}^L$ can be measured, and hence the voltage potentials $P_1^L,\ldots,P_{n_L}^L$ can be computed. In the remainder of this paper, when we refer to (measuring) the output $y(T)$ for some time $T$, we therefore mean the (measured) current $J^L(T)$.     
 Furthermore, note the similarity between the equations describing the hardware implementation~\eqref{eqn: hardware ANN} and the ANN equations~\eqref{eqn: neuron eqns}. 

 \begin{remark}\label{remark:CCVS}
In Figure~\ref{fig:ANN}, the activation function is denoted by a triangle labeled $\sigma$ and is implemented by a current-controlled voltage source. 
This can be done using the two circuits described in the following.
The first circuit consists of a combination of resistors and operational amplifiers (see \cite{Webster2010}). This circuit takes input current $\Bar{J}_k^{\ell}$ and generates an output voltage $X$ that is equal in value to  $\Bar{J}_k^{\ell}$, similar to \cite[Fig. 6]{Bayat2017}. Here, the operational amplifier ensures that $\Bar{P}_k^{\ell}=0$.  
The second circuit consists of exponential operational amplifiers, inverting amplifiers, summing amplifiers, and logarithmic amplifiers (see \cite{Webster2010}). This circuit generates an output voltage $P_k^\ell = \sigma(X)$. 
An example of this can be found in \cite[Fig. 1B]{Bao2017}, where the circuit gives output voltage potential 
    $v_0 = \tanh(v_i)$
for input voltage potential $v_i$.
The series interconnection of the two aforementioned  circuits then implements a current-controlled voltage source \cite{Desoer1969}.      
 \end{remark}

\begin{remark} \label{remark: pair of memristors}
    Note that the memductance, and hence the synaptic weight value, is always positive by Assumption~\ref{assumption 2}. The results in this paper can immediately be extended to implement both positive and negative synaptic weights. See also Section~\ref{sec: mnist example} for an example on this.     
    To implement negative weights, the weights $\hat{M}_{kj}^\ell$ are represented by a \textit{pair} of memristors rather than a single one \cite{Zamanidoost2015}, \cite{Alibart2013}. 
    The crossbar arrays need to be extended with $n_\ell$ rows so that the first $n_\ell$ rows correspond to a memductance matrix, say $W(\varphi^+)$, and the last $n_\ell$ rows correspond to a memductance matrix, say $W(\varphi^-)$, such that 
    \begin{equation} \label{eqn: pair of memristors}
        \hat{M}_{kj}^\ell  = W(\varphi_{kj}^{\ell,+}) - W(\varphi_{kj}^{\ell,-}).
    \end{equation} 
   Now, the first circuit described in Remark~\ref{remark:CCVS} needs to be altered such that it generates an output voltage $X$ that is equal in value to $\bar{J}_k^{\ell,+}-\bar{J}_k^{\ell,-}$ from input
   currents $\bar{J}_k^{\ell,+}$ and $\bar{J}_k^{\ell,-}$. Here, the input
   current $\bar{J}_k^{\ell,+}$ and $\bar{J}_k^{\ell,-}$ respectively refer to the currents at the end of row $k$ and $n+k$.
   This can be achieved by adding additional operational amplifiers to the circuit implementation, %resulting in a circuit 
   similar to \cite[Fig. 6]{Bayat2017}.
\end{remark}

\section{PROBLEM STATEMENT} \label{sec: problem statement}
The goal of this paper is to study the hardware implementation of ANNs as introduced in Section~\ref{sec:hardware implementation}. In particular, we investigate three problems. First, we consider the inference problem, which amounts to evaluating the ANN at a given input, using only input-output signals of its hardware implementation. 
Second, we study the reading problem, that is, retrieving the synaptic weight values of the ANN on the basis of input-output signals of the corresponding hardware implementation. 
Next, we study the writing problem, which concerns steering the synaptic weight values of the ANN, and hence the instantaneous memductance values of its hardware implementation, to desired values,  using 
input-output signals of the hardware implementation.

There are some difficulties to be wary of when solving the aforementioned problems. 
The first difficulty arises due to the tunability property of memristors. Namely, the memductance values change as the circuit is subjected to a nonzero input voltage, which means that the circuit implements a different ANN with different synaptic weight matrices $\hat{M}^\ell$ after a voltage signal is applied. This poses some difficulties in solving the inference problem. However, note that the very same property is exploited in the writing problem. 
The second difficulty stems from the memductance values not being measurable directly. Instead, information on the currents $\Bar{J}^\ell$ and the voltage potentials $P^{\ell-1}$ needs to be utilized to derive the memductance values $W(\varphi^\ell)$ as in \eqref{eqn: hardware ANN}. 
To this end, in the remainder of this paper, we assume that we can measure the currents $\Bar{J}^{\ell}$ for all $\ell\in[L]$. 
In addition, we assume the activation function $\sigma(\cdot)$ to be known.
With that, the voltage potentials $P^{\ell-1}$ can be derived as 
\begin{equation}\label{eqn: voltage potential}
    P^{\ell-1} = \sigma^{\ell-1}(\Bar{J}^{\ell-1}).
\end{equation}  
Here,
note that it is impossible to compute the memductance value for $P^{\ell-1}=0$.

Now, the inference problem can be formulated as follows.
\begin{problem}[Inference problem]\label{inference problem}
Let matrices $\hat{M}^\ell \in \mathds{R}^{n_\ell \times n_{\ell-1}}$, for $\ell \in [L]$, be given. For $S^\ell = 1$ for all $\ell\in[L]$, find $T > 0$ and encoding and decoding functions
\begin{align}
    \Lambda_{\text{enc}} : \hat{u}\mapsto u \quad \Lambda_{\text{dec}} : y\mapsto \hat{y}
\end{align}
with $u:[0,T]\rightarrow\mathds{R}^{n_0}$ and $y:[0,T]\rightarrow\mathds{R}^{n_L}$, in such a way that, for $\varphi^\ell_0$ satisfying $W(\varphi_0^\ell)= \hat{M}^\ell$, the solution to \eqref{eqn: hardware ANN} for $u(t) = \Lambda_{\text{enc}}(\hat{u})(t)$ is such that 
\begin{enumerate}
    \item $\Lambda_{\text{dec}}(y)=h(\hat{u})$;
    \item $\varphi^\ell(T; \varphi_0^\ell,P^{\ell-1}) =\varphi^\ell_0$ for all $\ell\in[L]$.
\end{enumerate} 
\end{problem}

Here, note that 1) implies that the output $\hat{y}= h(\hat{u})$ of the artificial neural network for some input $\hat{u}\in\mathds{R}^{n_0}$ can be inferred using measurements of the electrical circuit~\eqref{eqn: hardware ANN}. 
Recall that the memductance values of the memristors in the network change when external stimuli are applied. However, 2) implies that the experiment we conduct is \emph{non-invasive} in the sense that the memductance values before and after the experiment are the same, i.e., $W(\varphi^\ell(T)) = W(\varphi_0^\ell) = \hat{M}^\ell$. 
This implies that the input signal $u$ to the circuit is such that we are able to continue computing with the function $h(\cdot)$. 
Furthermore, note that the input signals $u$ correspond to voltage potentials $P_1^0,\ldots, P_{n_0}^0$, which can be chosen directly through voltage sources, and the voltage potentials $P_1^L,\ldots, P_{n_L}^L$ coincide with the corresponding outputs $y$.

In Problem~\ref{inference problem}, we have assumed that the desired weights $\hat{M}^\ell$ are given.
These weights would typically be obtained by training the ANN and need to be transferred to the hardware implementation~\eqref{eqn: hardware ANN}. 
To this end, a relevant problem is how to retrieve the values of $W(\varphi_0^\ell)$ on the basis of measurements of the external signals. 
We refer to this problem as \textit{reading}. 
Here, recall that the memductance values cannot be measured directly. 
The problem of reading the memductance matrices can now be formulated as follows.

\begin{problem}[Reading Problem] \label{reading problem}
Find $T\geq 0$, $P^0:[0,T]\rightarrow\mathds{R}^{n_0}$, and 
$S^\ell:[0,T]\rightarrow\{0,1\}^{n_\ell \times  n_{\ell-1}}$ for $\ell\in[L]$, such that, for any initial conditions $\varphi^1_0,\ldots,\varphi^L_0$, the following two properties hold for all $\ell\in[L]$:
\begin{enumerate}
    \item $\Bar{J}^\ell(t;\varphi_0^\ell,P^{\ell-1}) = \Bar{J}^\ell(t;\Bar{\varphi}_0^\ell,P^{\ell-1})$ for all $t\in[0,T]$ $\implies W(\varphi_0^\ell)=W(\Bar{\varphi}_0^\ell)$;
    \item $\varphi^\ell(T; \varphi_0^\ell,P^{\ell-1})=\varphi^\ell_0$.
\end{enumerate} 
\end{problem}

In other words, the reading problem is to choose the input voltage sources $P^0$ and the switch settings on the time-interval $[0,T]$ in such a way that the initial memductance matrices $W(\varphi^\ell_0)$ can be uniquely determined from the measured currents $\Bar{J}^{\ell}$.
Recall from the inference problem that the second item implies that $W(\varphi^\ell(T))=W(\varphi^\ell_0)=\hat{M}^\ell$.

Finally, to introduce the third and final problem studied in this paper, we note that the inference problem assumes that the memductance matrices are initialized precisely at the weights of the artificial neural network. In case not all memductances are initialized correctly, which can be verified with a solution to the reading problem, a relevant question is how to steer these memductances to their desired values. Recall that the desired memductance values would typically result from training the ANN. 
The issue is now formalized as follows, and is referred to as the \emph{writing} problem.

\begin{problem}[Writing problem] \label{writing problem}
Let $\epsilon>0$ and let $W_d^\ell\in\mathds{R}_{>0}^{n_\ell\times n_{\ell-1}}$ be the desired memductance matrix for $\ell\in[L]$. Find $T\geq 0$, $P^0:[0,T]\rightarrow\mathds{R}^{n_0}$, and $S^\ell:[0,T]\rightarrow\{0,1\}^{n_\ell \times  n_{\ell-1}}$  in such a way that 
    \begin{equation}
        \left|W^\ell_{d,kj} - \left[W(\varphi^\ell(T; \varphi_0^\ell,P^{\ell-1}))\right]_{kj}\right|\leq \epsilon
    \end{equation}
    for all $\ell\in[L], k\in [n_\ell],$ and $j\in [n_{\ell-1}] $.
\end{problem}

\begin{remark}
    As seen in \cite{Heidema2024}, the switches in the crossbar array are crucial in steering the memductances to desired values. 
For the reasons mentioned there, we have included a switch accompanying each memristor in the memristive crossbar arrays, contrary to, e.g., the arrays in  \cite{Yildiz2019}.
\end{remark}

\section{Inference} \label{sec:inference}
In order to perform the computation $\hat{y}=h(\hat{u})$ for some $\hat{u}$ in the circuit implementation of the ANN, we will consider block signals. Let $\tau>0$ and consider the block signal 
\begin{equation}
    Q(t) = \left\{
        \begin{array}{rl}
            0 & \quad t<-2\tau \\
            -1 & \quad t \in [-2\tau, -\tau)  \\ 
            1 & \quad t \in [-\tau, \tau] \\ 
            -1 & \quad t \in (\tau, 2\tau] \\
            0 & \quad t>2\tau.  
        \end{array}
    \right. \label{eqn: V reading}
\end{equation} 
A depiction of the block signal~\eqref{eqn: V reading} can be found in Figure~\ref{fig:block signal}. 
Note that the signal $Q(\cdot)$ is odd on the interval $[-2\tau,0]$ and odd on the interval $[0,2\tau]$. This motivates the following lemma.

\begin{figure}
    \centering
    \includegraphics[width=0.6\linewidth]{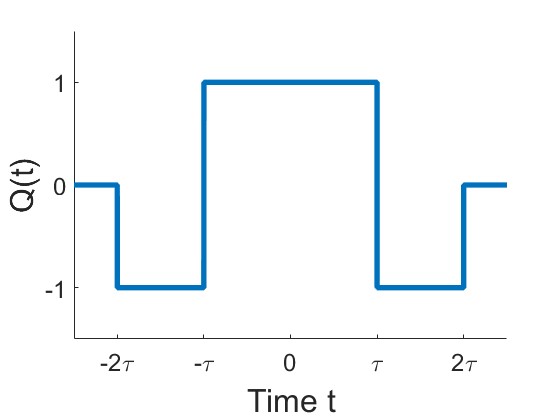}
    \caption{Block signal $Q(t)$ as in~\eqref{eqn: V reading}.}
    \label{fig:block signal}
\end{figure}

\begin{lemma}\label{lemma: odd voltage}
Consider the dynamics~\eqref{eqn: hardware ANN}. Let  $\ell\in[L]$ and $T>0$.  Let the switch matrix $S^\ell$ be constant for $t\in[0,T]$. 
    Furthermore, let $P^{\ell-1}:[0,T]\rightarrow \mathds{R}^{n_{\ell-1}}$ be 
    \begin{enumerate}
        \item odd on the interval  $[0,\frac{T}{2}]$, \label{property: odd on interval 1}
        \item odd on the interval  $[\frac{T}{2}, T]$. \label{property: odd on interval 2}
    \end{enumerate} 
    For each $\varphi_0^\ell\in\mathds{R}^{n_\ell\times n_{\ell-1}}$, $\varphi^\ell(\ \cdot \ ; \varphi_0^\ell,P^{\ell-1})$ is even on the intervals $[0,\frac{T}{2}]$ and $[\frac{T}{2}, T]$ and, moreover, $P^\ell= \sigma^\ell(\Bar{J}^\ell(\ \cdot \ ; \varphi_0^\ell,P^{\ell-1}))$ satisfies Properties~\ref{property: odd on interval 1}) and \ref{property: odd on interval 2}) above.     
\end{lemma}

The proof of Lemma~\ref{lemma: odd voltage} can be found in Appendix~\ref{Proof of odd voltage signals}. 
Note that Lemma~\ref{lemma: odd voltage} shows that certain properties of the voltage signal $P^\ell$ are preserved throughout all layers of the circuit implementation of the ANN, for $\ell\in[L]$.
This allows us to prove the following theorem, showing that the block signal \eqref{eqn: V reading} can be used to solve the inference problem.

\begin{theorem} \label{theorem: inference}
Let $T=4\tau$. Then, the encoding and decoding functions
\begin{align}
    \Lambda_{\text{enc}}(\hat{u}) &:=  \hat{u} Q(t-\tfrac{T}{2}) \quad \text{and} \quad 
    \Lambda_{\text{dec}}(y) := y\left(\tfrac{T}{2}\right) 
\end{align}  
     solve Problem~\ref{inference problem}. Here, $Q(\cdot)$ is the block function as in \eqref{eqn: V reading}. 
\end{theorem}

\begin{proof}
We have to show that items~\ref{property: odd on interval 1} and \ref{property: odd on interval 2} of Problem~\ref{inference problem} hold. To do so, consider $\ell\in[L]$, $k\in [n_\ell]$, $j\in~[n_{\ell-1}]$, and let $\varphi^\ell_0$ be such that $W(\varphi_0^\ell)=\hat{M}^\ell$.
As the function $t\mapsto Q(t-\tfrac{T}{2})$ is odd on the intervals $[0,\tfrac{T}{2}]$ and $[\tfrac{T}{2},T]$, it follows from Lemma~\ref{lemma: odd voltage} that the input 
$u(t) = \Lambda_{\text{enc}}(\hat{u})(t)$ is such that 
    \begin{align}
        \varphi_{kj}^\ell \left( \frac{T}{2}-t\right) &= \varphi_{kj}^\ell(t), \ \forall t\in\left[0,\tfrac{T}{4}\right], \\
         \varphi_{kj}^\ell \left( \frac{3T}{2}-t\right) &= \varphi_{kj}^\ell(t), \ \forall t\in\left[\tfrac{3T}{4}, T\right].
    \end{align}
 In particular, this implies that 
\begin{equation} \label{eqn: phi0=phiT}
    \varphi_{kj}^\ell(0) = \varphi_{kj}^\ell(\tfrac{T}{2}) = \varphi_{kj}^\ell(T).  
\end{equation}
Hence, it holds that 
$\varphi^\ell(T; \varphi_0^\ell,P^{\ell-1})=\varphi^\ell_0$, i.e., item~\ref{property: odd on interval 2} holds.

    Now, by the design of the encoding block signal, we have
    \begin{equation}\label{eqn: activation 0}
        P_j^0(\tfrac{T}{2}) = \hat{u}_j.
    \end{equation}
   Using \eqref{eqn: hardware ANN} and \eqref{eqn: phi0=phiT}, the resulting current at time $\tfrac{T}{2}$ is given by
   \begin{align}
        \Bar{J}_k^\ell(\tfrac{T}{2}) 
        &=  \sum_{j=1}^{n_{\ell-1}} W(\varphi_{kj}^\ell(0)) P_j^{\ell-1}(\tfrac{T}{2}).
    \end{align}
    Since this holds for any $k$ and $j$, and since  $W(\varphi_0^\ell)=\hat{M}^\ell$, we hence find that 
    \begin{equation}
        \Bar{J}^\ell (\tfrac{T}{2}) = \hat{M}^\ell P^{\ell-1} (\tfrac{T}{2}).
    \end{equation}
   The resulting voltage into layer $\ell$ is then given by
    \begin{align}
        P^\ell(\tfrac{T}{2}) = \sigma^\ell(\Bar{J}^{\ell}(\tfrac{T}{2})) = \sigma^\ell\left(\hat{M}^\ell P^{\ell-1}(\tfrac{T}{2})\right).
    \end{align} 
This implies that
\begin{align}
    P^0(\tfrac{T}{2}) &=  \hat{u}, \\
    P^{\ell}(\tfrac{T}{2}) &= \sigma^\ell\left(\hat{M}^\ell P^{\ell-1}(\tfrac{T}{2})\right), \quad \ell\in[L], \\
    \hat{y}&= P^L(\tfrac{T}{2}).
\end{align}
 Hence, we conclude that $\Lambda_{\text{dec}}(y)=\hat{y}=h(\hat{u})$.  
\end{proof}

Theorem~\ref{theorem: inference} hence tells us that the inference problem, i.e., Problem~\ref{inference problem}, can be solved using input voltage signals of the form \eqref{eqn: V reading}. 

\section{READING} \label{sec:reading}
In this section, we again make use of the block signal~\eqref{eqn: V reading} and prove that it can be used to solve the reading problem, i.e., Problem~\ref{reading problem}. 
To do so, we first need the following. 

\begin{definition} \label{definition: path}
    A \textit{path $\gamma$ to the $(k,j)$-th memristor in layer $\ell$} is defined as 
    $$ \gamma = (\gamma_0, \gamma_1, \ldots, \gamma_\ell) $$
    with $\gamma_{\ell-1}=j, \gamma_{\ell}=k$, and for which $\gamma_\kappa\in[n_\kappa]$.  
\end{definition}

To a path, we can associate the corresponding switch settings, defined next.  
\begin{definition} \label{definition: switches according to path}
    Consider a path $\gamma$ to the $(k,j)$-th memristor in layer $\ell$. The \textit{switch settings corresponding to the path $\gamma$} are defined in the following way: for all $\kappa\in[\ell]$, set
    \begin{equation}
    S_{ab}^\kappa =  \left\{
        \begin{array}{ll}
            1 & \quad \text{if } a=\gamma_\kappa, b=\gamma_{\kappa-1}, \\
            0 & \quad \text{otherwise}. 
        \end{array}
    \right.   
\end{equation}
 Furthermore, if $\ell<L$, set $S^{\ell+1}=0$ and set $S^\kappa$ arbitrarily for $\kappa\in\{\ell+2,\ldots,L\}$. Moreover, the $(a,b)$-th memristor in layer $\kappa\in[\ell]$ is said to be \textit{on the path $\gamma$} if $S_{ab}^\kappa=1$.
\end{definition} 

Thus, a path defines a single current path through the circuit implementing the ANN, in which only one memristor is selected in each layer up to layer $\ell$. Moreover, subsequent layers are completely disconnected. 

The reading of the instantaneous memductance value of the $(k,j)$-th memristor in layer $\ell$ now goes via the memristors in layers $\kappa\in[\ell-1]$ which are on the path $\gamma$. 
In particular, the input voltage $P_{\gamma_0}^0$ travels through the $(\gamma_\kappa, \gamma_{\kappa-1})$-th memristor in layer $\kappa\in[\ell]$.
 To simplify notation, let $\phi_\gamma^\kappa, P_\gamma^\kappa$ and $\Bar{J}_\gamma^\kappa$ respectively denote the flux $\varphi^\kappa_{\gamma_\kappa,\gamma_{\kappa-1}}$, the voltage potential $P^\kappa_{\gamma_{\kappa-1}}$ and the current $\Bar{J}_{\gamma_{\kappa}}^\kappa$ for all memristors on the path $\gamma$ and with $\kappa\in[\ell]$. 
 In fact, when the path is clear from context, we completely omit the subscript $\gamma$. Using \eqref{eqn: hardware ANN}, we can express the dynamics of the memristors on the path as 
\begin{equation}\label{eqn: dynamics writing}
 \begin{array}{ll} 
     \dot{\phi}_\gamma^1 &= f^1(\phi_\gamma,P_\gamma^0), \quad  \Bar{J}_\gamma^1 = W(\phi_\gamma^1)f^1(\phi_\gamma, P_\gamma^0), \\ 
     &  \vdots \hspace{2.9cm} \vdots
     \\
    \dot{\phi}_\gamma^\ell &= f^\ell(\phi_\gamma,P_\gamma^0), \quad \Bar{J}_\gamma^\ell = W(\phi_\gamma^\ell)f^\ell(\phi_\gamma, P_\gamma^0),
 \end{array} 
 \end{equation}
  where the functions $f^\kappa$ with $\kappa\in[\ell-1]$ are recursively defined as
  \begin{equation}\label{eqn: f^k}
      \begin{array}{ll}
            f^1(\phi_\gamma,P_\gamma^0) &= P_\gamma^0, \\
      f^{\kappa+1}(\phi_\gamma,P_\gamma^0) &= P_\gamma^\kappa = \sigma(W(\phi_\gamma^\kappa)f^\kappa(\phi_\gamma,P_\gamma^0)),
      \end{array}
  \end{equation} 
  with 
  $$\phi_\gamma := \begin{bmatrix}
          \phi_\gamma^1 & \hdots & \phi_\gamma^{\ell}
      \end{bmatrix}^\top.$$
      These dynamics are collected as 
      \begin{equation}\label{eqn: dynamics writing - simplified & compact}
          \begin{array}{ll}
              \dot{\phi}_\gamma &= f(\phi_\gamma,P_\gamma^0), \\
      \Bar{J}_\gamma &= \text{diag}(W(\phi_\gamma^1), \ldots, W(\phi_\gamma^\ell)) f(\phi_\gamma,P_\gamma^0),
          \end{array}
      \end{equation} 
      where
\begin{align} 
      f(\phi_\gamma,P_\gamma^0) &:= \begin{bmatrix}
          f^1(\phi_\gamma,P_\gamma^0) & \hdots & f^\ell(\phi_\gamma,P_\gamma^0)
      \end{bmatrix}^\top, \\
      \Bar{J}_\gamma &= \begin{bmatrix}
           \Bar{J}_\gamma^1 & \hdots & \Bar{J}_\gamma^\ell
      \end{bmatrix}^\top.
  \end{align} 
  Here, adopting the notation introduced in Section \ref{sec:hardware implementation}, we denote the solution $\phi_\gamma$ to \eqref{eqn: dynamics writing - simplified & compact} for initial condition $\phi_\gamma(0)=\phi_{\gamma,0}$ and input voltage $P_\gamma^0$ as $\phi_\gamma(t; \phi_{\gamma,0},P_\gamma^0)$. The corresponding current is denoted by $\Bar{J}_\gamma(t; \phi_{\gamma,0},P_\gamma^0)$.
  In particular, let $\Bar{J}_\gamma^\ell(t;{\phi}_{\gamma,0}^\ell,P_\gamma^{\ell-1})$ denote the $\ell$-th element of the current $\Bar{J}_\gamma(\cdot)$ and note that, using \eqref{eqn: f^k} and \eqref{eqn: dynamics writing - simplified & compact}, it is given by      
\begin{align}\label{eqn:J_gamma^l}
         \Bar{J}_\gamma^\ell(t;{\phi}_{\gamma,0}^\ell,P_\gamma^{\ell-1})  &= W(\phi_\gamma^{\ell}(t)) P_\gamma^{\ell-1}(t). 
     \end{align}

Now, the following algorithm solves the problem of reading the instantaneous memductance value of the $(k,j)$-th memristor in layer $\ell$. 

\begin{algorithm}[H]  
\caption{Reading the $(k,j)$-th memristor in layer $\ell$}
\label{algorithm reading - v2}
Let $T=4\tau$. 
 Fix a path $\gamma$ to the $(k,j)$-th memristor in layer $\ell$. 
\begin{algorithmic}[1]
\State Set $S^\kappa$ as in Definition~\ref{definition: switches according to path} for all $\kappa\in[L]$
\State Apply 
$$P_\gamma^0(t)=Q(t-\tfrac{T}{2}) \: \text{for}  \: t\in[0, T]$$ 
with $Q(\cdot)$ as in \eqref{eqn: V reading} 
\State Measure $\Bar{J}_\gamma^{\ell}(\tfrac{T}{2})$ and $\Bar{J}_\gamma^{\ell-1}(\tfrac{T}{2})$ \label{step: measure J's}
\State Compute $W(\varphi_{kj}^{\ell}(0)) = \Bar{J}_\gamma^{\ell}(\tfrac{T}{2})/ \sigma(\Bar{J}_\gamma^{\ell-1}(\tfrac{T}{2}))$ \label{step: compute W}
\end{algorithmic}
\end{algorithm}

The claim that Algorithm~\ref{algorithm reading - v2} can be used to solve the problem of reading the $(k,j)$-th memristor in layer $\ell$ is formalized in the following theorem.
\begin{theorem}\label{thm: reading}
Let $\ell\in[L], k\in [n_\ell]$, and $j\in[n_{\ell-1}]$, and let 
$$\gamma=(\gamma_0, \gamma_1,\ldots,\gamma_\ell)$$ 
    be any path to the $(k,j)$-th memristor in layer $\ell$. 
Then, for any initial condition $\varphi^\kappa(0) = \varphi_0^\kappa\in\mathds{R}^{n_\kappa \times n_{\kappa-1}}$, the input voltage and switch settings described in Algorithm~\ref{algorithm reading - v2} are such that $\varphi^\kappa(T;\varphi_0^\kappa,P^{\kappa-1})=\varphi^\kappa_0$, $\kappa\in[L]$. 
Moreover,
$\Bar{J}_\gamma^\ell(t;\phi_{\gamma,0}^\ell,P_\gamma^{\ell-1}) = \Bar{J}_\gamma^\ell(t;\Bar{\phi}_{\gamma,0}^\ell,P_\gamma^{\ell-1})$ for all $t\in[0,T]$ implies that $ W(\phi_{\gamma,0}^\ell)=W(\Bar{\phi}_{\gamma,0}^\ell)$, 
    with $\phi_{\gamma,0}^\ell = \varphi_{\gamma_\ell,\gamma_{\ell-1}}^\ell(0)$ and $\Bar{J}_\gamma^\ell(\cdot)$ as in \eqref{eqn: dynamics writing}.
\end{theorem}

\begin{proof}
    To simplify the notation in this proof, we let $\phi, \phi^\kappa,$ and $\phi_0^\kappa$ respectively denote $\phi_\gamma, \phi_\gamma^\kappa,$ and $\phi_{\gamma,0}^\kappa$ for $\kappa\in[\ell]$.  
    Now, consider any layer $\kappa\in[L]$ and any ($a,b$)-th memristor in that layer.
    If the memristor is not on the path $\gamma$, 
    there is no voltage over it and it trivially holds that  
    $$ \varphi_{ab}^\kappa(0) = \varphi_{ab}^\kappa(\tfrac{T}{2}) = \varphi_{ab}^\kappa(T). $$
    Alternatively, consider the memristor to be on the path $\gamma$. 
    As the function $t\mapsto Q(t-\tfrac{T}{2})$ is odd on the interval $[0,\tfrac{T}{2}]$ and odd on the interval $[\tfrac{T}{2},T]$, it follows from Lemma~\ref{lemma: odd voltage} that the input block signal $P_\gamma^0(t) = Q(t-\tfrac{T}{2})$ is such that 
\begin{equation} 
    \phi^\kappa(0) = \phi^\kappa(\tfrac{T}{2}) = \phi^\kappa(T).  \label{eqn:phi0=phiT}
\end{equation} 
 Hence, it follows that $\varphi^\kappa(T; \varphi_0^\kappa,P^{\kappa-1})=\varphi^\kappa_0$ for all $\kappa\in[L]$.
In particular, this implies that $W(\varphi^\kappa(T; \varphi_0^\kappa,P^{\kappa-1})) = W(\varphi_0^\kappa)= \hat{M}^\kappa$ for all $\kappa\in[L]$.

Now, consider any $\Bar{\phi}_0^\ell$ such that 
\begin{equation} \label{eqn: J=J}
    \Bar{J}_\gamma^\ell(t;\phi_0^\ell,P_\gamma^{\ell-1}) = \Bar{J}_\gamma^\ell(t;\Bar{\phi}_0^\ell,P_\gamma^{\ell-1})
\end{equation} 
for all $t\in[0,T]$.  
By \eqref{eqn:J_gamma^l}, \eqref{eqn:phi0=phiT}, and \eqref{eqn: J=J}, it then follows that 
\begin{equation} \label{eqn:J=J2}
    \begin{array}{ll}
         W(\phi_0^{\ell}) P_\gamma^{\ell-1}(\tfrac{T}{2}) &= \Bar{J}_\gamma^\ell(\tfrac{T}{2};{\phi}_0^\ell,P_\gamma^{\ell-1})  \\
      &= \Bar{J}_\gamma^\ell(\tfrac{T}{2};\Bar{\phi}_0^\ell,P_\gamma^{\ell-1})  \\
    &= W(\Bar{\phi}_0^{\ell}) P_\gamma^{\ell-1}(\tfrac{T}{2}).
    \end{array}
\end{equation}
Let us now show by induction that $P_\gamma^{{\ell-1}}(\tfrac{T}{2})\neq 0$. Hereto, note that, by definition, $P_\gamma^0(\tfrac{T}{2})=1\neq 0$. 
Let us now assume that $P_\gamma^\kappa(\tfrac{T}{2})\neq 0$ for some $\kappa\in\{0,\ldots,\ell-2\}$ and show that then also $P_\gamma^{\kappa+1}(\tfrac{T}{2})\neq 0$. 
By definition of the algorithm and using \eqref{eqn: f^k}, we have that 
\begin{align}
    P_\gamma^{\kappa+1}(\tfrac{T}{2}) 
    &= \sigma(W(\phi^{\kappa+1})  P_\gamma^\kappa(\tfrac{T}{2})).
\end{align}
Here, note that, by Assumption~\ref{assumption 1}, it holds that 
$$\sigma(x) = 0 \quad \iff \quad x=0.$$
Furthermore, since the memductance is always positive and $P_\gamma^{\kappa}(\tfrac{T}{2})$ is assumed to be nonzero, it is implied  that
\begin{equation}
    P_\gamma^{\kappa+1}(\tfrac{T}{2}) \neq 0.
\end{equation}  
By induction, it thus holds that $P_\gamma^{{\ell-1}}(\tfrac{T}{2})\neq 0$. 
Combining this with \eqref{eqn:J=J2} gives
\begin{equation}
\begin{array}{ll}
    W({\phi}^{\ell}_0) = W(\Bar{\phi}^{\ell}_0),
    \end{array}
\end{equation}
finalizing the proof.
\end{proof}

Note that the proof of Theorem~\ref{thm: reading} tells us that the instantaneous memductance values can be computed using only the measurements of the currents in step~\ref{step: measure J's} of Algorithm~\ref{algorithm reading - v2}, namely:
\begin{equation} 
    \begin{array}{ll}
        W(\phi_\gamma^{\ell}(0)) = \frac{\Bar{J}_\gamma^{\ell}(\tfrac{T}{2})}{P_\gamma^{\ell-1}(\tfrac{T}{2})},
    \end{array}
\end{equation}
where 
\begin{equation}
     P_\gamma^{\ell-1}(\tfrac{T}{2}) = \sigma(\Bar{J}_\gamma^{\ell-1}(\tfrac{T}{2})).
\end{equation}
In other words, 
\begin{equation} \label{eqn: compute M_kj^l}
        W(\varphi_{kj}^{\ell}(0)) = \tfrac{\Bar{J}_k^{\ell}(\tfrac{T}{2})}{\sigma(\Bar{J}_j^{\ell-1}(\tfrac{T}{2}))}
\end{equation}
as was already stated in step~\ref{step: compute W} of Algorithm~\ref{algorithm reading - v2}.
With this, we find that the reading problem, i.e., Problem~\ref{reading problem}, can be solved by applying Algorithm~\ref{algorithm: reading full}.

\begin{algorithm} 
\caption{Reading the memristor in all layers }
\label{algorithm: reading full}
\begin{algorithmic}[1]
\NoDo
\For{$\ell=1,2,\ldots,L$}
\For{$j=1,2,\ldots,n_{\ell-1}$} 
\For{$k=1,2,\ldots,n_\ell$}
\State  \parbox[t]{\dimexpr\textwidth-\leftmargin-\labelsep-\labelwidth}{Fix a path $\gamma$ to the $(k,j)$-th memristor in \\ layer $\ell$ and set  $S^\kappa$ as in Definition~\ref{definition: switches according to path} for \\ $\kappa\in[L]$} 
\State Set $t=0$
\EndFor
\EndFor
\EndFor
\end{algorithmic}
\end{algorithm}

Note that the algorithm is such that the memristors are read independently
and that their initial memductance values can be uniquely determined from the measured currents by Theorem~\ref{thm: reading}. 
Furthermore, the reading of the memristors is done in a non-invasive way, meaning that $W(\varphi_0^\ell)=W(\varphi^\ell(T))$ for all $\ell\in[L]$.
As a result, using Theorem~\ref{thm: reading}, we obtain the following: 
\begin{corollary}
    Algorithm~\ref{algorithm: reading full} solves the reading problem, i.e., Problem~\ref{reading problem}, within finite time $T\sum_{\kappa=0}^{L-1} n_\kappa n_{\kappa+1}$.
\end{corollary}

\begin{remark}
    Note that one can read all memristors on a path $\gamma$ simultaneously by measuring the currents $\Bar{J}_\gamma^\kappa$ for all $\kappa\in[\ell]$. Furthermore, one can read all $j$-th columns in all the layers at the same time by appropriate setting of the switches. 
    The latter would lead to Algorithm~\ref{algorithm: reading full} being done within finite time $T\max\{n_{\kappa-1} \ | \ \kappa\in[L]\}$.
\end{remark}

\section{WRITING} \label{sec:writing}
To solve the writing problem, we can extend the analysis done in \cite{Heidema2024}. Consider any layer $\ell\in[L]$ and let $\mathcal{D}^\ell$ denote the set of realizable memductance matrices of the layer, given by 
    \begin{align} 
        \mathcal{D}^\ell&=\left\{ \left. X \in \mathds{R}_{>0}^{n_\ell\times n_{\ell-1}}  \right\vert \exists \varphi^\ell \in \mathds{R}^{n_\ell\times n_{\ell-1}} \text{ s.t. } X = W(\varphi^\ell)  \right\}. \label{eqn: desired memductance matrices set}
    \end{align} 
     Note that this is the set of all memductance matrices that one could create for the layer, solely based on the limitations of each memristor separately without considering any restrictions imposed by the circuit dynamics~\eqref{eqn: hardware ANN}.  
     
Now, take any desired $\epsilon>0$ and, for each layer $\ell\in[L]$, consider the desired memductance matrix $W_d^\ell\in\mathcal{D}^\ell$. 
Furthermore, assume that the instantaneous memductance matrices $W(\varphi^\ell(0))$ are known for each layer. Note that Algorithm~\ref{algorithm reading - v2} can be used to obtain these matrices.

In our writing algorithm, the layers will be written one-by-one, starting from the last layer and going backwards to the first layer, and that, as soon as the layer $\ell$ is entirely written, it is `cut off' from the layers before it such that the writing process of layer $\ell-1$ does not interfere with the already written memristors in layers $\ell,\ldots,L$. In particular, this is done by setting $S^{\ell}=0$. Note that the switch settings associated to a path $\gamma$ to the $(k,j)$-th memristor in layer $\ell-1$, see Definition~\ref{definition: switches according to path}, ensure that $S^{\ell}=0$.  
Therefore, we will make use of this property of paths when solving the writing problem.

Let us now consider the problem of steering $W(\varphi_{kj}^\ell)$ to $W_{d,kj}^\ell$ for $\ell\in[L]$, $k\in[n_\ell]$, and $j\in[n_{\ell-1}]$.
Since $W_d^\ell\in\mathcal{D}^\ell$, we know that there exists some $\hat{\varphi}_d^\ell$ such that $W(\hat{\varphi}_d^\ell) = W_d^\ell$. In particular,  $W(\hat{\varphi}_{d,kj}^\ell) = W_{d,kj}^\ell$.
Due to strict monotonicity, see Assumption~\ref{assumption W strict monotone}, the problem of steering the memductance value is equivalent to the problem of steering $\varphi_{kj}^\ell$ to $\hat{\varphi}_{d,kj}^{\ell}$.
Now, similar to what was done in the reading algorithm, let us fix a path $\gamma$ to the $(k,j)$-th memristor in layer $\ell$, see Definition~\ref{definition: path}, set the switches accordingly, as specified in Definition~\ref{definition: switches according to path}, and consider the corresponding dynamics~\eqref{eqn: dynamics writing - simplified & compact}. 
To simplify notation, let $\hat{\phi}^\ell$ and $\hat{W}^\ell$ respectively denote $\hat{\varphi}_{d,kj}^{\ell}$ and $W(\hat{\varphi}_{d,kj}^{\ell})$.

Now, an algorithm is presented that solves the problem of steering $\phi_\gamma^\ell$ to $\hat{\phi}^{\ell}$. 
Here, the control input is based on the error between the desired and the current memductance value. 
The latter is computed using the measured currents in the circuit.

\begin{algorithm} 
\caption{Writing the $(k,j)$-th memristor in layer $\ell$ }
\label{algorithm writing 1 memristor}
Fix $T>0$, $\alpha>0$, and a path $\gamma$ to the $(k,j)$-th memristor in layer $\ell$. 
\begin{algorithmic}[1]
\NoDo
\State Set $S^\kappa$ as in Definition~\ref{definition: switches according to path} for $\kappa\in[L]$
\State Apply $P_\gamma^0(t)=x$ for any $x\neq 0$ to \eqref{eqn: dynamics writing - simplified & compact} for $t\in[0,T]$ \label{step: input}
 \State Measure $\Bar{J}_\gamma^\ell(T)$ and $\Bar{J}_\gamma^{\ell-1}(T)$  
 \State Set $i=1$
\While{$\left| \hat{W}^\ell - {\Bar{J}_\gamma^\ell(iT)}/{\sigma(\Bar{J}_\gamma^{\ell-1}(iT))} \right|> \epsilon$}
    \State \label{step: control input} \parbox[t]{\dimexpr\textwidth-\leftmargin-\labelsep-\labelwidth}{Apply $P_\gamma^0(t)=\alpha \left( \hat{W}^\ell -{\Bar{J}_\gamma^\ell(iT)}/{\sigma(\Bar{J}_\gamma^{\ell-1}(iT))} \right) $ to \\ \eqref{eqn: dynamics writing - simplified & compact} for $t\in(iT,(i+1)T]$}
    \State Measure $\Bar{J}_\gamma^\ell((i+1)T)$ and $\Bar{J}_\gamma^{\ell-1}((i+1)T)$  
    \State Set $i=i+1$ 
       \EndWhile 
\end{algorithmic}
\end{algorithm}

Recall that the current memductance value cannot be measured directly. Instead, the measured currents $\Bar{J}_\gamma^\ell$ and $\Bar{J}_\gamma^{\ell-1}$ are used to compute the memductance value. 
Here, note that $P_\gamma^{\ell-1}(iT)=\sigma(\Bar{J}_\gamma^{\ell-1}(iT))$ and hence  
$$\tfrac{\Bar{J}_\gamma^\ell(iT)}{\sigma(\Bar{J}_\gamma^{\ell-1}(iT))} = W(\phi_\gamma^\ell(iT)).$$  
Furthermore, note that initialization of the control input in step~\ref{step: input} of the algorithm with a non-zero value is necessary for the controller to be well-defined. Otherwise,  $\Bar{J}^\ell(0)$ would be zero and hence, by Assumption~\ref{assumption 1}, $\sigma(\Bar{J}^\ell(0))=0$. 
Hence, the control input given in step~\ref{step: control input} of Algorithm \ref{algorithm writing 1 memristor} is based on the error between the desired and the current memductance value.

The following theorem now formalizes the claim that the problem of steering $\phi_\gamma^\ell$ to $\hat{\phi}^\ell$ is solved, the proof of which can be found in Appendix~\ref{Proof of algorithm}. 
\begin{theorem}\label{theorem algorithm}
    Let $\epsilon>0$, $\ell\in[L]$, $k\in\ [n_\ell]$, $j\in [n_{\ell-1}]$, and $W_d^\ell\in\mathcal{D}^\ell$. Let $\alpha>0$ and $T>0$ be such that 
    \begin{equation} \label{eqn: choice of alpha and T}
      T\alpha \leq \frac{1}{\beta (\eta W_\mathrm{max})^{\ell-1}},
    \end{equation}
    with $\eta>0$, $W_\mathrm{max}>0$, and $\beta>0$ as in Assumptions~\ref{assumption 1} and \ref{assumption W strict monotone}.
    Then, for any $\varphi_0^\kappa\in\mathds{R}^{n_\kappa \times n_{\kappa-1}}$, $\kappa\in [\ell]$, there exists $\hat{T}_{kj}^\ell\geq 0$ such that the controller described in Algorithm~\ref{algorithm writing 1 memristor} achieves
    \begin{equation} \label{eqn: epsilon bound}
        |\hat{W}^\ell-W(\phi_\gamma^\ell(\hat{T}_{kj}^\ell))|\leq\epsilon, 
    \end{equation} 
     where $\phi_\gamma^\ell(\cdot)$ is the solution to \eqref{eqn: dynamics writing - simplified & compact}.    
\end{theorem}

With this, we find that the writing problem for the hardware implementation of the ANN, i.e., Problem~\ref{writing problem}, can be solved by applying the following algorithm. 
\begin{algorithm} 
\caption{Writing the memristor in all layers }
\label{writing full}
\begin{algorithmic}[1]
\NoDo
\For{$\ell=L, L-1,\ldots, 1$}
\For{$j=1,2,\ldots, n_{\ell-1}$} 
\For{$k=1,2,\ldots, n_\ell$}
\State  \parbox[t]{\dimexpr\textwidth-\leftmargin-\labelsep-\labelwidth}{Fix a path $\gamma$ to the $(k,j)$-th memristor in \\ layer $\ell$ }
\State Choose $\alpha>0$ and $T>0$ such that \eqref{eqn: choice of alpha and T} holds
\State  \parbox[t]{\dimexpr\textwidth-\leftmargin-\labelsep-\labelwidth}{Apply Algorithm~\ref{algorithm writing 1 memristor} and let $\hat{T}_{kj}^\ell\geq 0$ denote \\ the time such that \eqref{eqn: epsilon bound} holds,
    with $\phi_\gamma^\ell(\cdot)$ \\ the solution to \eqref{eqn: dynamics writing - simplified & compact}} \label{step: apply alg}
\State Set $t=0$
\EndFor
\EndFor
\EndFor
\end{algorithmic}
\end{algorithm}

Note that the time $\hat{T}_{kj}^\ell\geq 0$ as seen in step~\ref{step: apply alg} of Algorithm~\ref{writing full} is finite and exists by Theorem~\ref{theorem algorithm}. Furthermore, by choosing the switches appropriately, the algorithm steers each memristor independently to attain the desired memductance value by Theorem~\ref{theorem algorithm}. 
As a result, we obtain the following: 

\begin{corollary}
    Algorithm~\ref{writing full} solves the writing problem, i.e., Problem~\ref{writing problem}, within finite time $$\hat{T}^1_{11}+\hat{T}_{21}^1+\ldots+\hat{T}_{n_ 1,n_0}^1+\hat{T}_{11}^2+\ldots+\hat{T}^L_{n_L,n_{L-1}}.$$
\end{corollary}

\begin{remark}
    Note that, to write the entire implementation of the ANN, Algorithm~\ref{writing full} calls $\sum_{\ell=0}^{L-1} n_\ell n_{\ell+1}$ times on Algorithm~\ref{algorithm writing 1 memristor}. This can be reduced to $\sum_{\ell=0}^{L-1}\max\{n_\ell,n_{\ell+1}\}$ times by appropriate setting of the switches.  
    Namely, one can apply Algorithm~\ref{algorithm writing 1 memristor} to multiple memristors in a layer at the same time, when the memristors are not in the same column and row. This way, the memristors on the `diagonals' in a layer can all be written at once.
\end{remark}

\section{APPLICATIONS} \label{sec: applications}
\subsection{Application 1 - academic example} \label{sec:application 1}
In this section, let us consider the ANN shown in Figure~\ref{fig: example 1}. This ANN has two input neurons, one hidden layer with three neurons, and two output neurons, i.e., $L=2, n_0=2,n_1=3,$ and $n_2=2$. Let 
$$\sigma(x) = \tanh(x)$$ 
be the activation function and let the weight matrices be given by
\begin{equation} \label{eqn: application - Mhat}
    \hat{M}^1 =  \frac{1}{2} \begin{bmatrix}
    1 & 7 \\ 5 & 5 \\ 7 & 1
\end{bmatrix} \quad \text{and} \quad \hat{M}^2 =  \frac{1}{2} \begin{bmatrix}
    1 & 3 & 7\\ 7 & 2& 1
\end{bmatrix}.
\end{equation} 
Note that $\tanh(\cdot)$ is $\eta$-Lipschitz continuous for $\eta=1$.

\begin{figure}
\centering
\begin{tikzpicture}[scale=0.6]
\centering
  \matrix (m) [matrix of math nodes, row sep=1em, column sep=5em]
    {    &  |[node style sp]|  &   \\
      |[node style sp]|  &   &  |[node style sp]| \\
        &  |[node style sp]| &   \\
       |[node style sp]|  &     &|[node style sp]|\\
         &  |[node style sp]| &   \\ };

    { [start chain,node distance=5mm,                     every join/.style={->,matlab1}] \chainin (m-2-1); 
     { [start branch=i1_1] \chainin (m-1-2) [join={node[above,labeled] {\frac{1}{2}}}];}  
     { [start branch=i1_3] \chainin (m-5-2) [join={node[above,labeled] {\frac{7}{2}}}];}
    }

    { [start chain,node distance=5mm,                     every join/.style={->,matlab2}] \chainin (m-2-1);  
    { [start branch=i1_2] \chainin (m-3-2) [join={node[above,labeled] {\frac{5}{2}}}];}   
    }

    { [start chain,node distance=5mm,                     every join/.style={->,matlab3}] \chainin (m-2-1); 
     { [start branch=i1_3] \chainin (m-5-2) [join={node[above,labeled] {\frac{7}{2}}}];}
    }

     { [start chain,node distance=5mm,                     every join/.style={->,matlab4}] \chainin (m-4-1); 
    { [start branch=i1_1] \chainin (m-1-2) [join={node[below,labeled] {\frac{7}{2}}}];} 
     }

     { [start chain,node distance=5mm,                     every join/.style={->,matlab5}] \chainin (m-4-1); 
     { [start branch=i1_3] \chainin (m-3-2) [join={node[below,labeled] {\frac{5}{2}}}];}  
    }

    { [start chain,node distance=5mm,                     every join/.style={->,matlab6}] \chainin (m-4-1);   
     { [start branch=i1_3] \chainin (m-5-2) [join={node[below,labeled] {\frac{1}{2}}}];}
    }

    { [start chain,node distance=5mm,                     every join/.style={->,dashed,matlab1}] \chainin (m-1-2); 
     { [start branch=i1_1] \chainin (m-2-3) [join={node[above,labeled] {\frac{1}{2}}}];}   
    }

    { [start chain,node distance=5mm,                     every join/.style={->,dashed,matlab2}] \chainin (m-1-2);   
     { [start branch=i1_3] \chainin (m-4-3) [join={node[above,labeled] {\frac{7}{2}}}];}
    }

     { [start chain, node distance=5mm,  every join/.style={->,dashed,matlab3}] \chainin (m-3-2); 
    { [start branch=i1_1] \chainin (m-2-3) [join={node[below,labeled] {\frac{3}{2}}}];}   
    }

    { [start chain, node distance=5mm,  every join/.style={->,dashed,matlab4}] \chainin (m-3-2);  
     { [start branch=i1_3] \chainin (m-4-3) [join={node[below,labeled] {1}}];}
    }

    { [start chain,node distance=5mm,  every join/.style={->,dashed,matlab5}] \chainin  (m-5-2); 
    { [start branch=i1_1] \chainin (m-2-3)  [join={node[above,labeled] {\frac{7}{2}}}];}   
    }

    { [start chain,node distance=5mm,  every join/.style={->,dashed,matlab6}] \chainin  (m-5-2);    
     { [start branch=i1_3] \chainin (m-4-3) [join={node[above,labeled] {\frac{1}{2}}}];}
    }

    \node[text width=3cm] at (-4,-4.5) {\textit{input layer `0'}};
    \node[text width=3cm] at (0.5,-4.5) {\textit{hidden layer `1'}};
    \node[text width=3cm] at (5.5,-4.5) {\textit{output layer `2'}};
\end{tikzpicture}
\caption{ANN - to be implemented in hardware.}
\label{fig: example 1}
\end{figure}

Now, consider the circuit implementation of this ANN of the form as shown in Figure~\ref{fig:ANN}.
Here, assume that the flux-controlled memristors are described by
$$ q = g(\varphi) := 2 \varphi - \tfrac{1}{2} \log(\varphi^2+1) + \varphi \arctan(\varphi). $$
It can then be verified that the memductance curve of these memristors is given by
\begin{equation} \label{eqn: application - W funct}
    W(\varphi) = \frac{dg(\varphi)}{d\varphi} = 2+\arctan(\varphi).
\end{equation} 
The memductance values of the memristors hence lie in the interval 
$(W_\mathrm{min}, W_\mathrm{max}) = (-\tfrac{\pi}{2}+2, \tfrac{\pi}{2}+2)$.  
Note that all the values of the elements of matrices $\hat{M}^1$ and $\hat{M}^2$ are contained in this interval, meaning that  $\hat{M}^\ell\in\mathcal{D}^\ell$ for $\ell\in[2]$.
Furthermore, note that the function $W(\cdot)$ as in \eqref{eqn: application - W funct} is $\beta$-Lipschitz continuous for $\beta=1$.
Now, let the initial flux values of the memristors be given by 
$\varphi^1(0) = 0$ and $\varphi^2(0) = 0.$
Hence, the instantaneous memductance matrices are given by
\begin{equation} \label{eqn: application - W0}
    W(\varphi^1(0)) = W(\varphi^2(0))^\top = \begin{bmatrix}
    2 & 2 \\ 2 & 2 \\ 2 & 2
\end{bmatrix}.
\end{equation} 

Let us now use Algorithm~\ref{writing full} to steer the memductance matrices to the desired matrices, i.e., to steer $W(\varphi^\ell(0))$ as in \eqref{eqn: application - W0} to $\hat{M}^\ell$ as in \eqref{eqn: application - Mhat} for $\ell\in[2]$.  
Here, we choose $\epsilon = 0.05$ and $T=1$. Note that 
\begin{align}
    \min_{\ell \in [L]} \left\{ \tfrac{1}{T\beta (\eta W_{\mathrm{max}})^{\ell-1}} \right\} &= \tfrac{1}{T\beta \max\left\{1, (\eta W_{\mathrm{max}})^{L-1}\right\}}.
\end{align}
Therefore, if we choose 
$$\alpha = \tfrac{1}{\max\left\{1, ( \tfrac{\pi}{2}+2))\right\}} = \tfrac{2}{4+\pi} $$
then \eqref{eqn: choice of alpha and T} is satisfied for all $\ell\in[2]$.

Now, consider the path
$\gamma = \begin{bmatrix}
    1 & 3 & 2
\end{bmatrix}$
for steering the $(2,3)$-th memristor in layer $2$, for which the memductance evolves under Algorithm~\ref{writing full} as depicted in Figure~\ref{fig:WritingMem1and2}. 
Here, it can be verified that the memductance value is steered to (approximately) $0.55\in [\tfrac{1}{2}-\epsilon, \tfrac{1}{2}+\epsilon]$. The peak in the curve at the first iteration is a result of the initial voltage that is applied in step~\ref{step: input} of Algorithm~\ref{algorithm writing 1 memristor}. Namely, the initially supplied voltage, here chosen positive, steered the memductance value further away from the desired value. The negative control input based on the error as seen in step~\ref{step: control input} of the algorithm which was supplied thereafter then steered the memductance to the desired value. The sign change of the input voltage results in the peak in the memductance curve.

\begin{figure}
    \centering
    \includegraphics[width=0.95\linewidth]{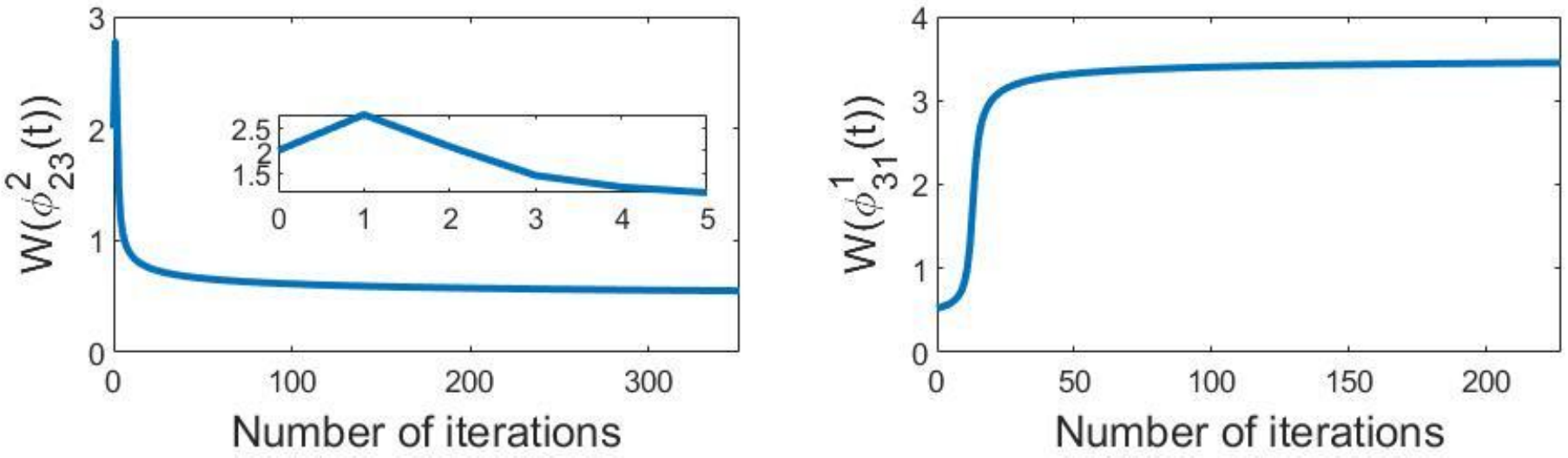}
    \caption{Writing - memductance curves of the $(2,3)$-th memristor in the 2nd crossbar array (left) and the $(3,1)$-th memristor in the 1st crossbar array (right).}
    \label{fig:WritingMem1and2}
\end{figure}

Similarly, all other memristors in layer $2$ can be written, after which this layer is cut off and we write the memristors in layer $1$. Hereto, consider the $(3,1)$-th memristor in layer $1$ and the corresponding path
$\hat{\gamma} = \begin{bmatrix}
    1 & 3
\end{bmatrix}.$
Let $\hat{T}$ denote the time that it took to steer all preceding memristors.
Simulating the writing algorithm in Matlab gives the memductance curve as seen in Figure~\ref{fig:WritingMem1and2}. 
Here, note that the starting point of the curve is not $2$, as one may have expected. Namely, since this memristor was on the path $\gamma$ for steering the $(2,3)$-th memristor in layer $2$, its memductance value has changed during that time. Hence, $W(\varphi_{31}^1(\hat{T}))\neq W(\varphi_{31}^1(0)) = 2$. 

Consider all the memristors to be written appropriately and let $T_w$ denote the time it took to steer all the memristors in both arrays of the circuit implementation of the ANN. We now want to do inference with the written circuit. Hereto, let us use the results in Section~\ref{sec:inference} to evaluate the ANN at the input 
$\hat{u}=\begin{bmatrix}
    -1 & 1
\end{bmatrix}^\top$. 
The corresponding voltage input signals $P_1^0$ and $P_2^0$ supplied to the circuit implementation are block signals with amplitude $-1$ and $1$, respectively, as specified in Theorem~\ref{theorem: inference}, where we take $\tau=5$.
Figures~\ref{fig:InferenceMem1} and~\ref{fig:InferenceMem2} show how the memductance values of the memristors in the two arrays change during the time the input signals are applied. Here, note that these plots show the memductance values to which the memristors were steered, as can be read out from their initial conditions. In particular, one can verify that the memductance matrices after writing are given by
\begin{align}
    W(\varphi^1(T_w))\! \approx\! \begin{bmatrix}
    0.55 \!&\! 3.45 \\ 2.55 \!&\! 2.55 \\ 3.45 \!&\! 0.55
\end{bmatrix}\!\!,\ 
W(\varphi^2(T_w)) \!\approx \!  \begin{bmatrix}
    0.55 \!&\! 3.45 \\
    1.53 \!&\! 1.04 \\
    3.45 \!&\! 0.55
\end{bmatrix}^{\!\top}\!\!. 
\end{align} 
Here, note that 
$|W(\varphi_{kj}^\ell(T_w))-\hat{M}_{kj}^\ell| \leq \epsilon$
for all $\ell\in[2], k\in[n_\ell],j\in[n_{\ell-1}]$. 
Furthermore, note that Figures~\ref{fig:InferenceMem1} and \ref{fig:InferenceMem2} show that the memductance value of each individual memristor is the same at time $t=T_w,T_w+10$, and $T_w+20$, as proven in Lemma~\ref{lemma: odd voltage}.

Now, the output voltage potentials at time $t=T_w+10$, corresponding to the input voltage signals, can be found to be (approximately) given by 
$$ P^2(T_w+10) \approx \begin{bmatrix}
    -0.99377 \\ 0.99374 
\end{bmatrix}. $$
Note that these values are close to the theoretical value
$$\sigma(\hat{M}^2\sigma(\hat{M}^1 \hat{u})) \approx \begin{bmatrix}
    -0.99380 \\ 0.99373 
\end{bmatrix},$$
where the values are not exactly the same as the memductance values of the memristors are not exactly equal to the desired values $\hat{M}^1$ and $\hat{M}^2$.

\begin{figure}
    \centering
    \includegraphics[width=0.95\linewidth]{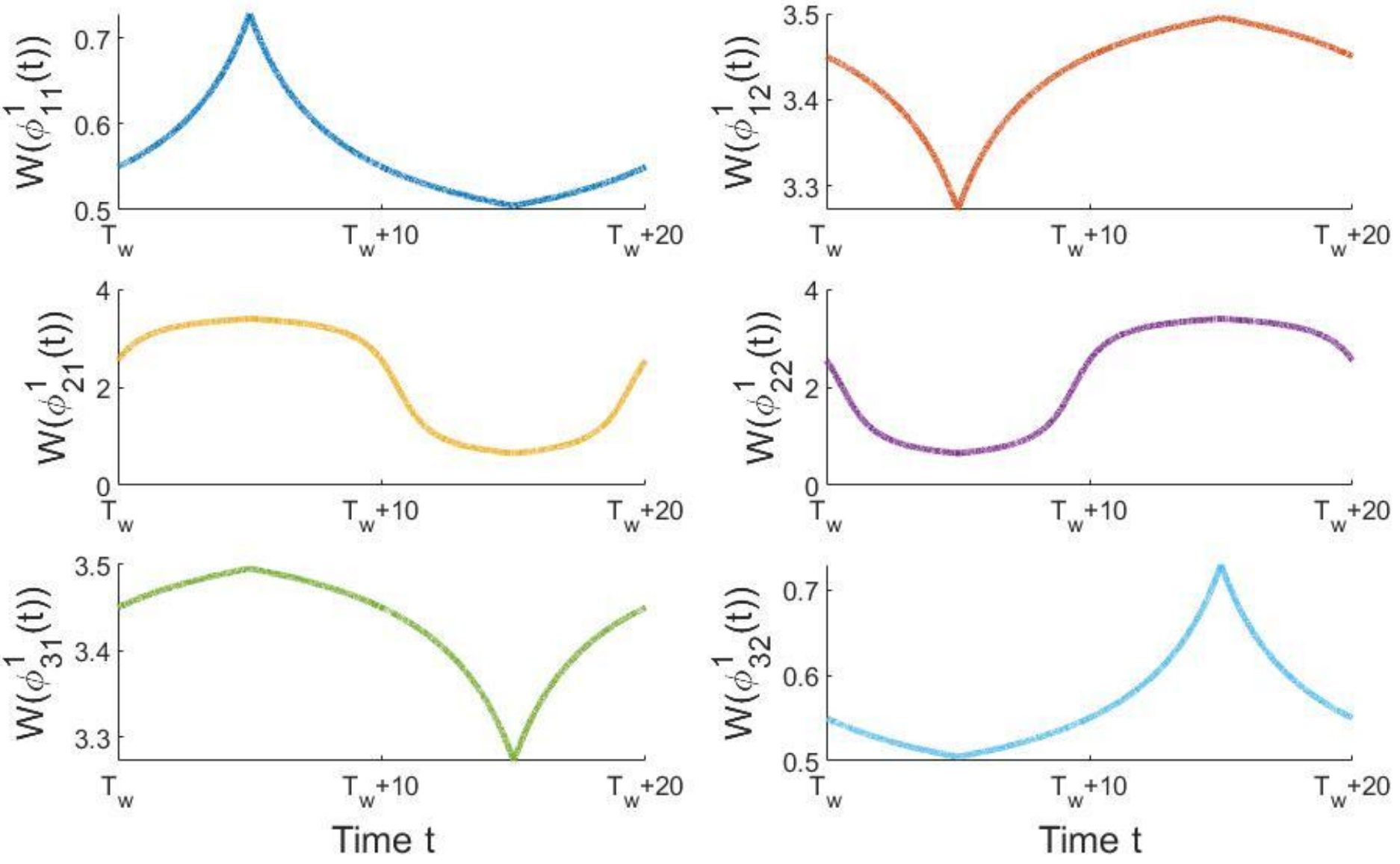}
    \caption{Inference - memductance curves of the memristors in the 1st crossbar array.}
    \label{fig:InferenceMem1}
\end{figure}

\begin{figure}
    \centering
    \includegraphics[width=0.97\linewidth]{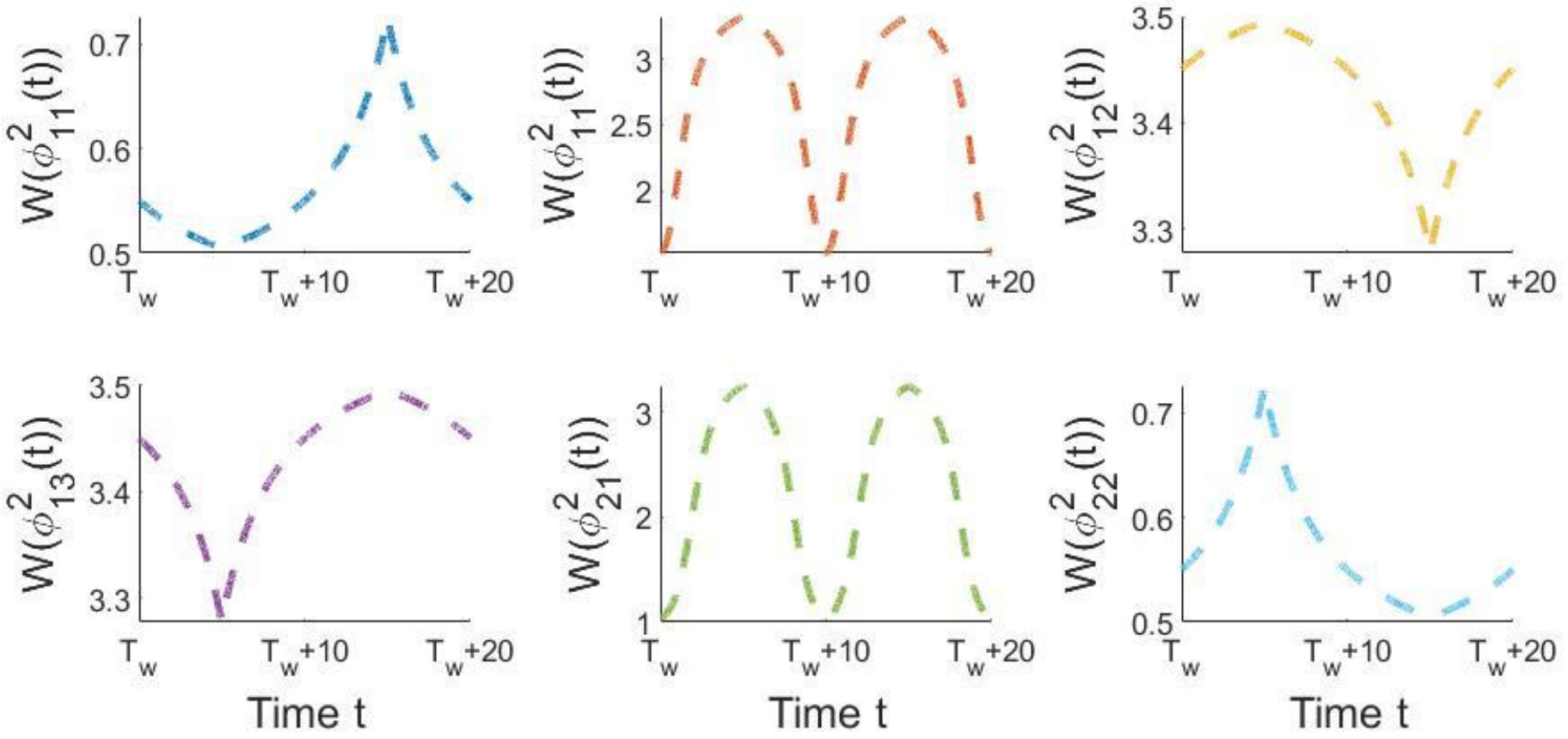}
    \caption{Inference - memductance curves of the memristors in the 2nd crossbar array.}
    \label{fig:InferenceMem2}
\end{figure}
\subsection{Application 2 - practical example}\label{sec: mnist example}
In this section, we apply the theory to an ANN that was trained with the MNIST data set.
The training dataset consists of 60000 images of handwritten digits 0,1,\ldots,9.
Each of these images is 28 by 28 pixels. 
The ANN that was trained to recognize the digits has $n_0=784$ input neurons, one for each pixel of the images, and $n_2=10$ output neurons, corresponding to digits 0-9. For simplicity, one hidden layer with $n_1=10$ neurons was added between the input and output layers. 
The activation function of the neurons was chosen as
$$ \sigma(x) = 3 \  \mathrm{sigmoid}(x)-1.5 = \tfrac{1.5-1.5 e^{-x}}{1+e^{-x}}. $$
Ideally, when the trained ANN is given an input image with the digit $k$, the output should be a vector of zeros with a `1' at the $(k+1)$-th element, $k\in\{0,1,\ldots,9\}$.
In our simulations, the maximum value in the output vector is considered to correspond to the digit recognized by the ANN. 
The trained ANN was tested using the MNIST test dataset containing 10000 new images of handwritten digits, for which an accuracy of approximately 88\% was obtained.

Now, the trained ANN with weight matrices $\hat{M}^1$ and $\hat{M}^2$ is implemented in the electrical circuit model \eqref{eqn: hardware ANN}. Here, note that the obtained weights are not all positive values. Instead, we also have negative weights. In particular, the elements of $\hat{M}^1$ and $\hat{M}^2$ approximately belong to the intervals $(- 1.8640,  1.8539)$, respectively, $(- 0.5586,  0.5140)$. 
  To implement negative weights, we make use of  Remark~\ref{remark: pair of memristors}. In particular, we consider the same flux-controlled memristors as in Section~\ref{sec:application 1}. 
Now, we implement each synaptic weight by a pair of memristors as 
$$\hat{M}_{kj}^\ell = W(\varphi_{kj}^{\ell})S_{kj}^{\ell} - W(\varphi_{k+n_\ell,j}^{\ell})S_{k+n_\ell,j}^{\ell},$$
for $\ell\in[2]$. The circuit implementation hence allows for any synaptic weight value in the interval
$ (-W_\mathrm{max}, W_\mathrm{max}) = (-\tfrac{\pi}{2}-2, \tfrac{\pi}{2}+2)$. 
Note that all the values of the elements of $\hat{M}^\ell$ are contained in this interval, meaning that $\hat{M}^\ell\in\mathcal{D}^\ell$ for $\ell\in[2]$. 
Furthermore, note that the initial flux values of the memristors corresponding to this implementation are given by
$ \varphi^\ell_0 = \tan(W(\varphi^\ell(0)) - 2) $ 
for each $\ell\in[2]$. 
Following Remark~\ref{remark: pair of memristors}, the activation function is implemented such that
$ P_k^{\ell}(t) = \sigma(\Bar{J}_k^\ell - \Bar{J}_{k+n_{\ell}})$, 
for all $\ell\in[2]$ and $k\in[n_\ell]$. 

Now, one can do inference with the circuit. In particular, one can take an image from the MNIST test dataset and feed it as voltage input to the circuit. Here, the inputs $P_j^0$ are block signals with an amplitude equal to the value of the corresponding pixel of the test image, as specified in Theorem~\ref{theorem: inference}, where we take $\tau=5$. 
If we consider the first image in the dataset, corresponding to a handwritten number 7, then the output voltage potentials at the mid-time of the block signal, i.e., at time $t=10$, can be found to be close to the theoretical value:
\begin{align}
    P^2(10) &\approx -\tfrac{1}{100} \begin{bmatrix}
    1 \!&\!  
  1 \!&\!
  1 \!&\! 
  1 \!&\! 
  1 \!&\! 
  1 \!&\! 
   0 \!&\! 
   -100 \!&\!
   -1 \!&\! 
   -2
\end{bmatrix}^\top \\
&\approx \sigma(\hat{M}^2\sigma(\hat{M}^1 \hat{u})).
\end{align}
In both the ANN and its circuit implementation, the recognized digit is hence correctly found to be the number 7.  
One can feed all images in the dataset to the circuit and verify that it recognizes the digits correctly with an accuracy of 88\%, which is the same accuracy as the ANN that it implements.

\section{CONCLUSIONS} \label{sec: conclusion}
We have shown how artificial neural networks can be implemented in hardware using flux-controlled memristors and memristive crossbar arrays with switches.
First, we have defined the \textit{inference} problem. This problem concerns evaluating an ANN at a given input by appropriately choosing the input voltage signals to the hardware implementation and measuring the corresponding output voltage potentials at a certain time. 
Furthermore, we have defined the problems of \textit{reading} and \textit{writing} the memductance values of the memristors in the circuit implementation of the ANN, i.e., the problems of determining and steering the memductance values. The reading problem was solved 
by means of an algorithm which determines the memductance values of the memristors one by one. In the algorithm, input voltage signals and switch settings are chosen in such a way that the memductance value of a memristor can be determined from corresponding current measurements.
The writing problem was solved by means of an algorithm in which input voltage signals and switch settings are updated by current measurements in such a way that the memductance values are steered towards desired values, one memristor at a time.
The results were then applied to two applications. The first application was an academic example which shows the writing of a network and evaluations of the written network thereafter. The second example showed a more practical application of the theory discussed in this paper, namely, to use the suggested circuit for image processing. 
In this example, an ANN, which was trained with the MNIST data set, was implemented in a circuit and evaluated at test data from the MNIST data set.

Future work will focus on sparse ANNs and how to efficiently deal with the sparsity in the circuit implementation of the ANN.

\section{APPENDIX} \label{sec: appendix}
\subsection{Proof of Lemma~\ref{lemma: odd voltage}} \label{Proof of odd voltage signals}
    Consider any $\ell\in[L]$ and $k\in [n_\ell]$. Using \eqref{eqn: hardware ANN}, we have  
        $P_k^\ell = \sigma(\Bar{J}_k^{\ell}),$
    with 
    \begin{equation}
        \bar{J}_k^\ell = \sum_{j=1}^{n_{\ell-1}} W(\varphi_{kj}^{\ell})S_{kj}^{\ell}P_j^{\ell-1} = \sum_{j=1}^{n_{\ell-1}} I_{kj}^{\ell}S_{kj}^{\ell},
    \end{equation}   
    where $I_{kj}^{\ell} = W(\varphi_{kj}^{\ell})P_j^{\ell-1}$. 
Note that, by Assumption~\ref{assumption 1}, we have 
    $\sigma(x)=0$ if and only if $x=0$.
In addition, note that the zero function trivially satisfies Properties~\ref{property: odd on interval 1} and \ref{property: odd on interval 2}. 
Now, consider any $j\in~[n_{\ell-1}]$. Then, by  \eqref{eqn: hardware ANN}, we obtain the dynamics
$\tfrac{d}{dt} \varphi_{kj}^\ell = S_{kj}^\ell P_j^{\ell-1}$, 
where $S_{kj}^\ell$ is assumed to be constant.
Then, $P_j^{\ell-1}$ being odd on the interval  $[0,\tfrac{T}{2}]$ means that 
\begin{equation} \label{eqn: oddness of P^{l-1}}
    P_j^{\ell-1}(t) = -P_j^{\ell-1}\left(\tfrac{T}{2}-t\right), \quad \forall t\in\left[0,\tfrac{T}{4}\right].
\end{equation}
Let $t\in\left[0,\tfrac{T}{4}\right]$ and consider 
\begin{align}
   & \varphi_{kj}^{\ell}\left(\tfrac{T}{2}-t\right) = \varphi_{kj}^{\ell}(0)+S_{kj}^\ell\int_0^{\tfrac{T}{2}-t} P_j^{\ell-1}(r) dr \\
    &\ = \varphi_{kj}^{\ell}(0)+S_{kj}^\ell\int_0^{\tfrac{T}{4}} P_j^{\ell-1}(r) dr +S_{kj}^\ell\int_{\tfrac{T}{4}}^{\tfrac{T}{2}-t} P_j^{\ell-1}(r) dr.
\end{align}
Noting that 
\begin{align}
  \int_{\tfrac{T}{4}}^{\tfrac{T}{2}-t}\!\! P_j^{\ell-1}(r) dr \!=\!  \int_{t}^{\tfrac{T}{4}}\!\! P_j^{\ell-1}\left(\tfrac{T}{2}-\Tilde{r}\right) d\Tilde{r} \!=\! -\!\!\int_{t}^{\tfrac{T}{4}}\!\! P_j^{\ell-1}(\Tilde{r}) d\Tilde{r}
\end{align}
by substitution of $\Tilde{r}=\tfrac{T}{2}-r$ and  \eqref{eqn: oddness of P^{l-1}}, we obtain
$\varphi_{kj}^{\ell}\left(\tfrac{T}{2}-t\right)  = \varphi_{kj}^{\ell}(t)$.
Hence, $\varphi_{kj}^{\ell}$ is even on the interval  $[0,\tfrac{T}{2}]$. 
In turn, this implies that 
\begin{align}
    &I_{kj}^{\ell}\left(\tfrac{T}{2}-t\right) = W\left(\varphi_{kj}^{\ell}\left(\tfrac{T}{2}-t\right)\right) P_j^{\ell-1}\left(\tfrac{T}{2}-t\right) \\
    & \ = -W\left(\varphi_{kj}^{\ell}\left(t\right)\right) P_j^{\ell-1}\left(t\right) = -I_{kj}^{\ell}(t).
\end{align}
Hence, $I_{kj}^{\ell}$ is odd on the interval  $[0,\tfrac{T}{2}]$.
Similarly, one can show that Property~\ref{property: odd on interval 2} implies that $\varphi_{kj}^{\ell}$ is even and $I_{kj}^{\ell}$ is odd on the interval  $[\tfrac{T}{2},T]$.
Now, since the sum of odd functions is again odd, this implies that $\bar{J}_k^\ell$ satisfies Properties~\ref{property: odd on interval 1} and \ref{property: odd on interval 2}.  
Using that the activation function is odd by Assumption~\ref{assumption 1}, it immediately follows that $P_k^\ell$ also satisfies Properties~\ref{property: odd on interval 1} and \ref{property: odd on interval 2}, finalizing the proof.
\subsection{Properties of the functions $f^\kappa$} \label{Proof of properties of f^k}
In this section, we will consider a constant input $P_\gamma^0$ on the time interval $[0,T]$ for some $T>0$. 
With this, we will state and prove some properties of the functions $f^\kappa$ as in \eqref{eqn: f^k} on the interval $[0,T]$.  
\begin{lemma} \label{lemma: properties of f^k}
    Let $f^\kappa$ with $\kappa\in[\ell-1]$ be as in \eqref{eqn: f^k}. 
    Then, the following properties hold: for all $\phi_\gamma$,
    \begin{enumerate}
        \item $\text{sign}(f^\kappa(\phi_\gamma,P_\gamma^0)) = \text{sign}(P_\gamma^0)$. \label{property 1 of fk}
        \item $|f^\kappa(\phi_\gamma,P_\gamma^0)| \leq (\eta W_\mathrm{max})^{\kappa-1} |P_\gamma^0|$. \label{property 2 of fk}
        \item $|f^{\kappa+1}(\phi_\gamma,P_\gamma^0)| \geq \sigma(W_\mathrm{min}|f^\kappa(\phi_\gamma,P_\gamma^0)|)$.   \label{property 3 of fk}
    \end{enumerate}
\end{lemma}

\begin{proof}
    To simplify the notation in this proof, let $\phi$ and $\phi^\kappa$ respectively denote $\phi_\gamma$ and $\phi_\gamma^\kappa$ for $\kappa\in[\ell]$.
    Let us first show that Property~\ref{property 1 of fk} holds. For $\kappa=1$, the property trivially holds. Now,  assume that Property~\ref{property 1 of fk} holds for some $\kappa\in[\ell-2]$. 
    Then, the strict monotonicity of $\sigma(\cdot)$, see Assumption~\ref{assumption 1}, and the positivity of the memductance can be used to show that 
    \begin{align}
        &\text{sign}(f^{\kappa+1}(\phi,P_\gamma^0)) = \text{sign}(\sigma(W(\phi^\kappa)f^\kappa(\phi,P_\gamma^0))) \\
        &= \text{sign}(W(\phi^\kappa)f^\kappa(\phi,P_\gamma^0))  = \text{sign}(f^\kappa(\phi,P_\gamma^0))  = \text{sign}(P_\gamma^0).
    \end{align}
    Hence, by induction, Property~\ref{property 1 of fk} holds.

    Let us now show that Property~\ref{property 2 of fk} holds.  For $\kappa=1$, we have that 
    $ |f^1(\phi,P_\gamma^0)| = |P_\gamma^0|$
    and hence the property trivially holds. Now,  assume that the property holds for some $\kappa\in[\ell-2]$. The Lipschitz continuity of $\sigma(\cdot)$, see Assumption~\ref{assumption 1}, and Assumption~\ref{assumption W strict monotone} can then be used to obtain that 
    \begin{align}
        &|f^{\kappa+1}(\phi,P_\gamma^0)| \leq \eta W_\mathrm{max} |f^\kappa(\phi,P_\gamma^0)| \leq (\eta W_\mathrm{max})^\kappa |P_\gamma^0|.
    \end{align}
     Hence, by induction, Property~\ref{property 2 of fk} holds.

     Lastly, let us show that Property~\ref{property 3 of fk} holds. Note that, since $\sigma(\cdot)$ is odd by Assumption~\ref{assumption 1}, we have 
     $ |\sigma(x)| = \sigma(|x|)$
     for any $x$. 
     Using Assumption~\ref{assumption W strict monotone} and that $\sigma(\cdot)$ is strictly monotone, see Assumption~\ref{assumption 1}, it follows that, for any $\kappa\in[\ell-2]$, 
     $|f^{\kappa+1}(\phi,P_\gamma^0)| \geq \sigma(W_\mathrm{min}|f^\kappa(\phi,P_\gamma^0)|)$,
    concluding the proof.
\end{proof}

Using Lemma~\ref{lemma: properties of f^k}, we can state and prove the following.

\begin{lemma}\label{lemma: exists ! solution to ODE}
     For any constant $P_\gamma^0$ and any initial condition $\phi_\gamma(0) = \phi_{\gamma,0}$, there exists a unique solution $\phi_\gamma$  to the ODE~\eqref{eqn: dynamics writing - simplified & compact}.
\end{lemma} 
\begin{proof}  
Consider any $\phi_{\gamma,0}$ and $P_\gamma^0$. 
To simplify the notation in this proof, let $\phi, \phi^\kappa,$ and $P^0$ respectively denote $\phi_\gamma, \phi_\gamma^\kappa,$ and $P_\gamma^0$ for $\kappa\in[\ell]$.
Now, a solution to the ODE~\eqref{eqn: dynamics writing - simplified & compact} always exists and is unique if the function  $f(\phi, P^0)$ is Lipschitz continuous with respect to $\phi$, i.e., for any choice of the constant input $P^0$, there exists $\mu>0$ such that
 \begin{equation} \label{eqn: Lipschitz f}
     \|f(\phi, P^0) - f(\Tilde{\phi}, P^0)\|_2 \leq \mu \|\phi -\Tilde{\phi}\|_2
 \end{equation} 
    for all $\phi, \Tilde{\phi}$.
    Here, $\|\cdot\|_2$ denotes the Euclidean norm.
    Let us now show that the function  $f(\phi, P^0)$ is Lipschitz continuous with respect to $\phi$. 

If $P^0=0$, then 
    $f^1(x, P^0) = P^0 = 0$
    for all $x$. Since $\sigma(0)=0$ by Assumption~\ref{assumption 1}, one can easily verify that 
    $f^\kappa(x, P^0) = 0$
    for all $\kappa\in[\ell-1]$ and for all $x$. In that case, \eqref{eqn: Lipschitz f} trivially holds for any choice of $\mu>0$. 
            
    Therefore, let us consider $P^0\neq 0$. 
    Note that, by norm equivalence, we know that there exists some $c>0$ such that     
    \begin{align}
        &\|f(\phi, P^0) - f(\Tilde{\phi}, P^0)\|_2  \leq c \|f(\phi, P^0) - f(\Tilde{\phi}, P^0)\|_1 \\
       &= c\sum_{\kappa=1}^{\ell-1} |f^{\kappa+1}(\phi, P^0) - f^{\kappa+1}(\Tilde{\phi}, P^0)|. \label{eqn: norm of f}
    \end{align}

    Now, take any $\kappa\in[\ell-1]$. 
    We claim that there exists $\mu_\kappa>0$ such that 
    \begin{equation}\label{eqn: fk lipschitz}
        |f^{\kappa+1}(\phi, P^0) - f^{\kappa+1}(\Tilde{\phi}, P^0)| \leq \mu_\kappa \|\phi -\Tilde{\phi}\|_2
    \end{equation} 
    for any $\phi$ and $\Tilde{\phi}$. For $\kappa=1$, we can use the Lipschitz continuity of $\sigma(\cdot)$ and $W(\cdot)$, see Assumptions~\ref{assumption 1} and \ref{assumption W strict monotone}, to obtain that
    \begin{align}
       &|f^2(\phi,P^0)-f^2(\Tilde{\phi},P^0)| \leq \eta |W(\phi^1)-W(\Tilde{\phi}^1)| |P^0| \\
        &\leq \eta \beta |P^0| |\phi^1-\Tilde{\phi}^1| =: \mu_1  |\phi^1-\Tilde{\phi}^1|. 
    \end{align}
    Note here that $\mu_1>0$ since $\eta,\beta>0$ by definition and $P^0\neq 0$ by assumption.  
Now, assume that the property holds for some $\kappa\in[\ell-2]$. 
Define
$$\mu_\kappa := 2\max\{\eta W_\mathrm{max} \mu_{\kappa-1}, \eta^\kappa \beta W_\mathrm{max}^{\kappa-1}|P^0|\} $$
and note that 
$ \| x \|_2  = \sqrt{x_1^2 + \ldots + x_m^2} \geq \sqrt{x_a^2} = |x_a| $
for any vector $x\in\mathds{R}^m$ and any $a\in[m]$.
Then, from Property~\ref{property 2 of fk} of Lemma~\ref{lemma: properties of f^k} and the Lipschitz continuity of $\sigma(\cdot)$ and $W(\cdot)$ one can verify that
\begin{align}
    &|f^{\kappa+1}(\phi,P^0)-f^{\kappa+1}(\Tilde{\phi},P^0)| 
       \leq \mu_\kappa \| \phi - \Tilde{\phi} \|_2. 
\end{align}
Using that $P^0\neq 0$ by assumption, it is clear that $\mu_\kappa>0$ by definition. 
 Hence, by induction, we have shown that there exists a positive $\mu_\kappa$ such that \eqref{eqn: fk lipschitz} holds for all $\phi$ and $\Tilde{\phi}$. Using \eqref{eqn: norm of f}, we hence find that 
 $$\|f(\phi, P^0) - f(\Tilde{\phi}, P^0)\|_2 \leq c \sum_{\kappa=1}^{\ell-1} \mu_\kappa \|\phi -\Tilde{\phi}\|_2 $$
 and therefore that \eqref{eqn: Lipschitz f} holds for 
 $ \mu = c \sum_{\kappa=1}^{\ell-1} \mu_\kappa, $
 which finalizes the proof.
\end{proof}
\subsection{Proof of Theorem~\ref{theorem algorithm}} \label{Proof of algorithm}
In this section, we will prove Theorem~\ref{theorem algorithm}, for which we will make use of the third statement of Theorem 13.9 in \cite{Haddad2008}. Hereto, note that a function $h:[0,a)\rightarrow[0,\infty)$ with $a\in(0,\infty]$ is called a class $\mathcal{K}$ function if $h(0)=0$ and it is strictly monotone. 
Now, let us define some notation. 
Let $\Phi(\phi_{\gamma,0},P_\gamma^0)$ denote the solution $\phi_\gamma(T)$ to \eqref{eqn: dynamics writing - simplified & compact} for initial condition $\phi_\gamma(0)=\phi_{\gamma,0}$ and (constant) input $P_\gamma^0$. Note that such a solution to the ODE always exists and is unique by  Lemma~\ref{lemma: exists ! solution to ODE}.  
With this notation and the controller in Algorithm~\ref{algorithm writing 1 memristor}, the dynamics is given by
\begin{equation}
    \phi_\gamma((i+1)T) = \Phi\left(\phi_\gamma(iT),P_\gamma^0(iT)\right)
\end{equation}
with 
\begin{equation}
    P_\gamma^0(iT) =  \left \{ \begin{array}{ll}
        x & \text{if } i=1,  \\
        \alpha(W(\hat{\phi}_\gamma^\ell)-W(\phi_\gamma^\ell(iT))) & \text{if } i\geq 2,
    \end{array} \right.
\end{equation} 
where $x\neq 0$. 
In what follows, we will denote
\begin{equation} \label{eqn: F(phi)}
    F(\phi_\gamma) = \Phi\left(\phi_\gamma,\alpha(W(\hat{\phi}_\gamma^\ell)-W(\phi_\gamma^\ell))\right).
\end{equation}

If we now define
$ \phi_{\gamma}^{[\ell-1]} := \begin{bmatrix}
          \phi_\gamma^1 & \hdots & \phi_\gamma^{\ell-1}
      \end{bmatrix}^\top $
  and   $ F^{[\ell-1]}(\phi_\gamma) :=\begin{bmatrix}
          F^1(\phi_\gamma)  & \hdots & F^{\ell-1}(\phi_\gamma)
      \end{bmatrix}^\top$, 
  then we have the following discrete-time nonlinear autonomous dynamical system 
   \begin{equation} \label{eqn: dynamics Theorem 13.9}
        \begin{array}{ll}
            \phi_\gamma^{[\ell-1]}((i+1)T) &= F^{[\ell-1]}(\phi_\gamma(iT)), \\ 
            \phi_\gamma^\ell((i+1)T) &= F^\ell(\phi_\gamma(iT)), 
        \end{array}
    \end{equation} 
  with initial conditions 
  \begin{equation}
     \phi_\gamma^{[\ell-1]}(0) = \begin{bmatrix}
          \phi^1_{\gamma, 0} & \hdots & \phi^{\ell-1}_{\gamma,0}
      \end{bmatrix}^\top \quad \text{and} \quad \phi_\gamma^\ell(0)=\phi^\ell_{\gamma,0}.
  \end{equation}
  Here, $F^k(\cdot)$ denotes the $k$-th element of the vector $F(\cdot)$ as seen in \eqref{eqn: F(phi)}, $k\in[\ell]$. 
  Note that, by definition, we have
\begin{equation} \label{eqn: F^l-phi^l}
    F^\ell(\phi_\gamma) - \phi_\gamma^\ell = \int_{0}^{T}    \dot{\phi}_\gamma^\ell(s) ds  = \int_{0}^{T}  f^\ell(\phi_\gamma(s),P_\gamma^0) ds,
\end{equation}
with $\phi_\gamma(\cdot)$ the solution to \eqref{eqn: dynamics writing - simplified & compact} for initial condition $\phi_\gamma^\ell$ and input
$P_\gamma^0 = \alpha(W(\hat{\phi}^\ell)-W(\phi_\gamma^\ell))$. 
From this, it follows that 
$$F^\ell\left(\begin{bmatrix}
      \phi_\gamma^{[\ell-1]} \\ \hat{\phi}_\gamma^\ell
  \end{bmatrix}\right)=\hat{\phi}^\ell$$ 
  for every $\phi_\gamma^{[\ell-1]}$. Note that this last statement implies that the desired flux value $\hat{\phi}^\ell$ is an equilibrium, i.e.,  if $\phi_\gamma^\ell(iT)=\hat{\phi}^\ell$, then the algorithm will stop and no voltage will be applied, meaning that 
$\phi_\gamma^\ell(jT)=\hat{\phi}^\ell, \quad \forall j\geq i.$ 

With this, we can apply the third statement of Theorem~13.9 in \cite{Haddad2008}. This statement is repeated here for convenience, albeit slightly modified to our notation.

\begin{proposition}[{\cite[Thm. 13.9]{Haddad2008}}] \label{thm: 13.9}
    Consider the system \eqref{eqn: dynamics Theorem 13.9}. Assume that there exists a continuous function $\xi:\mathds{R}^{\ell-1}\times \mathds{R}\rightarrow \mathds{R}$ and class $\mathcal{K}$ functions $\theta(\cdot)$ and $\gamma(\cdot)$ satisfying
    \begin{align}
        \xi\left(\begin{bmatrix}
      \phi_\gamma^{[\ell-1]} \\ \hat{\phi}^\ell
  \end{bmatrix}\right) =0,  \quad &
        \theta(|\phi_\gamma^\ell-\hat{\phi}^\ell|) \leq \xi(\phi_\gamma), \text{ and} \\
        \xi(F(\phi_\gamma))  -\xi(\phi_\gamma) &\leq -\gamma(|\phi_\gamma^\ell-\hat{\phi}^\ell|),  
    \end{align}
    for all $\phi_\gamma^{[\ell-1]}\in\mathds{R}^{\ell-1}$ and $\phi_\gamma^\ell\in\mathds{R}$. 
Then, the nonlinear dynamical system given by \eqref{eqn: dynamics Theorem 13.9} is asymptotically stable with respect to $\phi_\gamma^\ell-\hat{\phi}^\ell$,  i.e., 
$\lim_{i\rightarrow \infty} \phi_\gamma^\ell(iT) =\hat{\phi}^\ell. $
\end{proposition}

Now, consider the candidate Lyapunov function 
$$\xi(\phi_\gamma) = (\phi_\gamma^\ell-\hat{\phi}^\ell)^2.$$ 
 This function is clearly equal to zero if $\phi_\gamma^\ell=\hat{\phi}^\ell$. 
In addition, $\xi(\cdot)$ satisfies 
$\xi(\phi_\gamma) =  \theta(|\phi_\gamma^\ell-\hat{\phi}^\ell|),$
where $\theta(\cdot)$ is the class $\mathcal{K}$ function given by
$\theta(x) =x^2$ for $x\in[0,\infty)$.
Furthermore, we have that
\begin{equation}
\begin{split}
\Delta \xi(\phi_\gamma) &:= \xi(F^\ell(\phi_\gamma))  -\xi(\phi_\gamma) \\
    &= (F^\ell(\phi_\gamma) - \phi_\gamma^\ell)^2 +2  (F^\ell(\phi_\gamma) - \phi_\gamma^\ell)(\phi_\gamma^\ell - \hat{\phi}^\ell).
\end{split}
\end{equation}
Note that, by Property~\ref{property 1 of fk} of Lemma~\ref{lemma: properties of f^k}, the sign of $f^\ell(\phi_\gamma(t),P_\gamma^0)$ is constant over the time interval $[0,T]$ and 
\begin{align}
    &\text{sign}(F^\ell(\phi_\gamma) - \phi_\gamma^\ell) =\text{sign}(f^\ell(\phi_\gamma(t),P_\gamma^0)) \\
    &= \text{sign}(P_\gamma^0) = - \text{sign}(\phi_\gamma^\ell-\hat{\phi}^\ell).
\end{align}
Here, the latter equality follows from the strict monotonicity of $W(\cdot)$, see Assumption~\ref{assumption W strict monotone}. Hence,
\begin{align}
    \Delta \xi (\phi_\gamma) &= |F^\ell(\phi_\gamma) - \phi_\gamma^\ell|^2 - 2|F^\ell(\phi_\gamma) - \phi_\gamma^\ell| |\phi_\gamma^\ell-\hat{\phi}^\ell|
\end{align}
Using \eqref{eqn: F^l-phi^l}, Property~\ref{property 2 of fk} of Lemma~\ref{lemma: properties of f^k}, and the monotonicity of $W(\cdot)$, then gives
\begin{align}
    &|F^\ell(\phi_\gamma) - \phi_\gamma^\ell| 
    \leq \int_{0}^{T} (\eta W_\mathrm{max})^{\ell-1} |P_\gamma^0|  ds \\
    &= T(\eta W_\mathrm{max})^{\ell-1} |\alpha(W(\hat{\phi}^\ell)-W(\phi_\gamma^\ell))| \\
    &\leq \alpha \beta T (\eta W_\mathrm{max})^{\ell-1} |\hat{\phi}^\ell - \phi_\gamma^\ell| \leq |\hat{\phi}^\ell - \phi_\gamma^\ell|.
\end{align}
Here, we have used \eqref{eqn: choice of alpha and T}. 
 Equivalently,
$- |\hat{\phi}^\ell - \phi_\gamma^\ell| \leq - |F^\ell(\phi_\gamma) - \phi_\gamma^\ell|$
such that 
$\Delta \xi(\phi_\gamma) \leq  -|F^\ell(\phi_\gamma) - \phi_\gamma^\ell|^2$.

Now, by Property~\ref{property 3 of fk} of Lemma~\ref{lemma: properties of f^k} we have that
$$ |f^{\kappa+1}(\phi_\gamma,P_\gamma^0)| \geq \sigma(W_\mathrm{min} |f^{\kappa}(\phi_\gamma,P_\gamma^0)|) $$
for all $\kappa\in[\ell-1]$. Hence, 
\begin{align}
    &|F^\ell(\phi_\gamma) - \phi_\gamma^\ell| =  \int_{0}^{T}  |f^\ell(\phi_\gamma(s),P_\gamma^0)| ds \\
    &\geq \int_{0}^{T} \sigma (W_\mathrm{min} \sigma(\ldots \sigma(W_\mathrm{min} |f^1(\phi_\gamma(s),P_\gamma^0)|))) ds \\
    &= T \sigma (W_\mathrm{min} \sigma(\ldots \sigma(W_\mathrm{min} |P_\gamma^0|))). 
\end{align}
Here, note that $|P_\gamma^0| = \alpha |W(\hat{\phi}^\ell)-W(\phi_\gamma^\ell)|$. 
Let us define 
$$\nu(r) := \inf_{\Tilde{\phi}_\gamma^\ell} \left\{ \alpha|W(\hat{\phi}^\ell)-W(\Tilde{\phi}_\gamma^\ell)| \ \Bigg| \ |\hat{\phi}^\ell - \Tilde{\phi}_\gamma^\ell| \geq r  \right \}$$
for $r\geq 0$. 
Then, it is clear that 
$\nu(|\hat{\phi}^\ell-\phi_\gamma^\ell|) \leq \alpha |W(\hat{\phi}^\ell)-W(\phi_\gamma^\ell)|$
for all $\phi_\gamma^\ell$.
Furthermore, it can easily be verified that $\nu(\cdot)$ is a class $\mathcal{K}$ function.  Therefore, from the monotonicity of $\sigma(\cdot)$, it follows that 
\begin{align}
    &|F^\ell(\phi_\gamma) - \phi_\gamma^\ell| \geq \gamma(|\hat{\phi}^\ell-\phi_\gamma^\ell|),
\end{align} 
where $\gamma(|\hat{\phi}^\ell-\phi_\gamma^\ell|) \coloneqq T \sigma (W_\mathrm{min} \sigma(\ldots \sigma(W_\mathrm{min} \nu(|\hat{\phi}^\ell-\phi_\gamma^\ell|))))$.
From $\sigma(\cdot)$ being strictly monotone and $\nu(\cdot)$ being a class $\mathcal{K}$ function,  it immediately follows that $\gamma(\cdot)$ is a class $\mathcal{K}$ function. 
Hence, we have found the class $\mathcal{K}$ function $\gamma(\cdot)$ such that 
$$\Delta \xi (\phi_\gamma) \leq - \gamma(|\hat{\phi}^\ell-\phi_\gamma^\ell|).$$
From Proposition~\ref{thm: 13.9}, it then follows that 
$\lim_{i\rightarrow \infty} \phi_\gamma^\ell(iT) =\hat{\phi}^\ell. $
Therefore, the controller described in Algorithm~\ref{algorithm writing 1 memristor} is such that \eqref{eqn: epsilon bound} holds for some $\hat{T}_{kj}^\ell\geq 0$, which proves the theorem.

\bibliographystyle{IEEEtran}
\bibliography{References}

\end{document}